\documentclass[11pt,a4paper,reqno]{amsart}

\usepackage{a4wide}
\usepackage{amssymb,amsmath,amsthm,mathrsfs}
\usepackage{graphicx}
\usepackage{float}
\usepackage{colortbl}
\usepackage{enumerate}
\usepackage{upref}
\usepackage{stackrel}
\usepackage{enumitem}
\usepackage[bookmarks=false]{hyperref}
\usepackage[T2A,T1]{fontenc}
\DeclareSymbolFont{cyrillic}{T2A}{cmr}{m}{n}
\DeclareMathSymbol{\D}{\mathalpha}{cyrillic}{196}

\usepackage{accents}
\newlength{\dhatheight}
\newcommand{\doublehat}[1]{%
    \settoheight{\dhatheight}{\ensuremath{\hat{#1}}}%
    \addtolength{\dhatheight}{-0.35ex}%
    \hat{\vphantom{\rule{1pt}{\dhatheight}}%
    \smash{\hat{#1}}}}

\theoremstyle{plain}% default
\newtheorem{theorem}{Theorem}[section]
\newtheorem{lemma}[theorem]{Lemma}

\newtheorem{proposition}[theorem]{Proposition}
\newtheorem{corollary}[theorem]{Corollary}

\theoremstyle{definition}
\newtheorem{definition}{Definition}[section]

\newtheorem{example}{Example}[section]
\newtheorem*{condition}{Condition}

\theoremstyle{remark}
\newtheorem{remark}[theorem]{Remark}%[section]

\usepackage{enumitem}
\makeatletter
\def\namedlabel#1#2{\begingroup
   #2%
 \def\@currentlabel{#2}%
   \phantomsection\label{#1}\endgroup
}

\def\R{\ensuremath{\mathbb R}}
\def\N{\ensuremath{\mathbb N}}

\def\I{\ensuremath{{\bf 1}}}
\def\e{{\ensuremath{\rm e}}}
\def\exp{{\ensuremath{\rm Exp}}}

\def\U{\ensuremath{\mathcal U}}
\def\B{\ensuremath{\mathcal B}}

\def\leb{{\rm Leb}}

\def\p{\ensuremath{\mathbb P}}

\def\A{\ensuremath{A^{(q)}}}
\def\n{\ensuremath{n}}

\def\X{\mathcal{X}}

\def\1{{\bf 1}}

\def\W{\ensuremath{\mathscr W}}
\def\BB{\ensuremath{\mathscr B}}
\def\AA{\ensuremath{\mathscr A}}

\def\ie{{\em i.e.}, }

\def\E{\mathbb E}

\def\dist{\ensuremath{\text{dist}}}

\def\TF{\mathcal{T}}

\def\eps{\varepsilon}

\def\cv{\ensuremath{\text {Cor}}}
\newcommand{\dif}{\mathrm{d}}

\numberwithin{equation}{section}

\begin{document}

\title[Complete convergence and records for dynamical systems]{Complete convergence  and records for dynamically generated stochastic processes}

\author[A. C. M. Freitas]{Ana Cristina Moreira Freitas}
\address{Ana Cristina Moreira Freitas\\ Centro de Matem\'{a}tica \&
Faculdade de Economia da Universidade do Porto\\ Rua Dr. Roberto Frias \\
4200-464 Porto\\ Portugal} \email{\href{mailto:amoreira@fep.up.pt}{amoreira@fep.up.pt}}
\urladdr{\url{http://www.fep.up.pt/docentes/amoreira/}}

\author[J. M. Freitas]{Jorge Milhazes Freitas}
\address{Jorge Milhazes Freitas\\ Centro de Matem\'{a}tica \& Faculdade de Ci\^encias da Universidade do Porto\\ Rua do
Campo Alegre 687\\ 4169-007 Porto\\ Portugal}
\email{\href{mailto:jmfreita@fc.up.pt}{jmfreita@fc.up.pt}}
\urladdr{\url{http://www.fc.up.pt/pessoas/jmfreita/}}

\author[M. Magalh\~aes]{M\'ario Magalh\~aes}
\address{M\'ario Magalh\~aes\\ Centro de Matem\'{a}tica da Universidade do Porto\\ Rua do
Campo Alegre 687\\ 4169-007 Porto\\ Portugal} \email{mdmagalhaes@fc.up.pt}

\thanks{MM was partially supported  by FCT grant SFRH/BPD/89474/2012, which is supported by the program POPH/FSE.
ACMF and JMF were partially supported by FCT projects FAPESP/19805/2014 and PTDC/MAT-CAL/3884/2014. All authors were partially supported by CMUP (UID/MAT/00144/2013), which is funded by FCT with national (MEC) and European structural funds through the programs FEDER, under the partnership agreement PT2020. ACMF and JMF would like to thank ICTP, where the final writing of the paper took place, for the financial support and hospitality. We thank Mike Todd for helpful comments and suggestions.}

\date{\today}


%37A50  	Relations with probability theory and stochastic processes
%60G70  	Extreme value theory; extremal processes
%37B20  	Notions of recurrence
%60G10  	Stationary processes
%37C25  	Fixed points, periodic points, fixed-point index theory
%--------
%37A25  	Ergodicity, mixing, rates of mixing
%37D25  	Nonuniformly hyperbolic systems (Lyapunov exponents, Pesin theory, etc.)
%37D35  	Thermodynamic formalism, variational principles, equilibrium states
%37C40  	Smooth ergodic theory, invariant measures
%60G55   	Point processes

\begin{abstract}
We consider empirical multi-dimensional Rare Events Point Processes that keep track both of the time occurrence of extremal observations and of their severity, for stochastic processes arising from a dynamical system, by evaluating a given potential along its orbits. This is done both in the absence and presence of clustering. A new formula for the piling of points on the vertical direction of bi-dimensional limiting point processes, in the presence of clustering, is given, which is then generalised for higher dimensions. The limiting multi-dimensional processes are computed for systems with sufficiently fast decay of correlations. The complete convergence results are used to study the effect of clustering on the convergence of extremal processes, record time and record values point processes. An example where the clustering prevents the convergence of the record times point process is given. 
\end{abstract}

\maketitle
\tableofcontents

\section{Introduction}

The main goal of this work is to study the complete convergence of multi-dimensional
empirical Rare Events Point Processes (REPP) arising from stationary stochastic processes generated by chaotic dynamical systems. Namely, we consider a deterministic discrete time dynamical system $(\X,\B,\mu,T)$, where $\X$ is a compact manifold, $\mathcal B$ is its
Borel $\sigma$-algebra, $T:\X\to\X$ is a measurable map and $\mu$ is a $T$-invariant probability measure, \ie $\mu(T^{-1}(B))=\mu(B)$, for all $B\in\B$. One can think of $T:\X\to\X$ as the evolution law that establishes how time affects the transitions from one state in $\X$ to another. We consider an observable (measurable) function $\varphi:\X\to\R$ and define the stochastic process $X_0, X_1, \ldots$ by:
\begin{equation}
\label{eq:SP-def}
X_n=\varphi \circ T^n, \quad \mbox{for every $n\in\N_0$,}
\end{equation}
where $T^n$ denotes the $n$-fold composition of $T$, with the convention that $T^0$ is the identity map on $\X$ and $\N_0$ is the set of non-negative integers.

We will be particularly interested in the two-dimensional empirical REPP counting the number of observations $X_j$ that lie on some normalised interval of thresholds, for $j$ on a certain normalised time frame, \ie 
\begin{equation}
\label{EREPP-def1}
N_n([t_1,t_2)\times[\tau_1,\tau_2))=\#\{j\in [nt_1,nt_2): u_n(\tau_2)<X_j\leq u_n(\tau_1)\},
\end{equation} 
where the sequence of levels $(u_n(\tau))_{n\in\N}$ is chosen to avoid non-degeneracy and is typically such that the expected number of exceedances, \ie observations satisfying $X_j> u_n(\tau)$, among the first $n$ random variables (r.v.) of the process, is asymptotically $\tau$ (see equation \eqref{un} below). %, as in the Poisson limit theorem for a sequence of binomially distributed random variables. 
In particular, this means that the event $X_j>u_n(\tau)$ is asymptotically rare, \ie $\lim_{n\to\infty}\mu(X_j>u_n(\tau))=0$, which motivates the name used for the empirical point process.

Since these processes keep information about both the times and magnitudes of extreme observations, they are quite useful for risk assessment and evaluation of the impact of hazardous events corresponding to phenomena that can be easily modelled by dynamical systems, such as meteorological events or the behaviour of financial markets. In fact, the complete convergence of such two-dimensional empirical REPP is a powerful technique to study the extremal behaviour of the systems and has a large scope of applications.  Using the appropriate projecting map, it allows us to obtain asymptotic behaviour of the higher order statistics and exceedances of high levels, analyse the behaviour of record times, \ie the moments at which a record observation occurs (\cite{P71,R75,R87}), or the convergence of the related extremal processes introduced in \cite{D64} and studied in \cite{L64, P71, R75}, or the convergence of the one-dimensional REPP as considered in \cite{HHL88, FFT10,FFT13}.

In the case of independent and identically distributed (iid) sequences of random variables or for stationary sequences for which there is no clustering of exceedances, \ie the non-appearance of clumps of abnormal observations above high thresholds, then the empirical REPP $N_n$ given as in \eqref{EREPP-def1} converges to a two-dimensional Poisson process, with two-dimensional Lebesgue measure as intensity measure (see \cite{P71, R75, R87} and \cite[Chapter~5.7]{LLR83}).

In the context of dynamical systems, we recall that a connection between extreme values and hitting times was observed by Collet \cite{C01} and formally established in \cite{FFT10,FFT11}. The idea is that the observable function $\varphi$ is typically maximised at a certain point $\zeta\in\X$ and so, observations with abnormally high values correspond to entrances of the orbit in small neighbourhoods of $\zeta$ and the magnitude of the excesses is determined by how small is the distance of the orbital point to $\zeta$. For higher dimensional systems, it may be useful to consider, instead of the two-dimensional time-magnitude space $[0,+\infty)\times[0,+\infty)$ used in \eqref{EREPP-def1}, the multi-dimensional time-position space $[0,+\infty)\times T_\zeta\X$, where $T_\zeta\X\subset\R^d$, in the second component, is the tangent space to $\X$ at $\zeta$, so that after a normalisation and projection from $T_\zeta\X$ to $\X$, we can keep track of the position of the orbit on the phase space, which will ultimately allow us to recover information about the distance to $\zeta$ and therefore about the magnitude of the observations. 
%$[0,+\infty)\times \V$, where $\V\subset\R^d$, in the second component, is identified with an open neighbourhood of $\zeta$ on the tangent space $T_\zeta\X$, so that after a normalisation and projection from $T_\zeta\X$ to $\X$, we can keep track of the position of the orbit on the phase space.

In the context of dynamical systems, the complete convergence of two-dimensional empirical REPP has been addressed in the very recent work \cite{HT17}, only in the absence of clustering of rare events, which means that the limiting process is a Poisson process, as in the independent case. In  order to prove the complete convergence, the authors use a thinning technique, as in \cite[Section~5.5--5.7]{LLR83}, that in some sense allows us to reduce the problem the one-dimensional multilevel case. Let us mention that, independently, in \cite{PS17}, the authors also study some general spatio-temporal point processes with a different construction (the spatial component is bounded), which they manage to apply to hyperbolic systems such as billiards. 

In this paper we prove the complete convergence of multi-dimensional empirical REPP both in the presence and absence of clustering of rare events for non-uniformly expanding dynamical systems which present a sufficiently fast loss of memory. One of the major highlights is the precise description we provide of the limiting processes that appear in the presence of clustering. We recall that, in this context, clustering is associated to some sort of underlying periodicity \cite{FFT12}, such as when the point $\zeta$, where $\varphi$ is maximised, is a periodic point. The effect of clustering in the limiting process for \eqref{EREPP-def1} can be portrayed as a piling of points in the vertical direction, corresponding to the observations within the same cluster, while in the horizontal direction these vertical piles appear scattered as in a typical Poisson process. This effect has been described in \cite{M77, H87, N02}. However, the distribution of the points on the vertical piles has never been clearly depicted nor computed. In fact, as stated in \cite{N02}, the Mori-Hsing characterisation of the limiting process, in \cite{M77, H87}, can be regarded as implicit. Using some sort of adaptation of the thinning technique, described in \cite[Section~5.5]{LLR83}, to the presence of clustering, more insight about the  limiting processes is given in \cite{N02}, by means of multilevel projections to 1-dimensional REPP. However, up to our knowledge,  there is still no clear picture about these processes in the literature. Hence, in our opinion, the simple representations we give of the limiting processes such as in equations \eqref{eq:repelling-limit-REPP-2d}, \eqref{eq:repelling-limit-REPP-multi-d}, \eqref{eq:repelling-limit-REPP-2d-special}, \eqref{eq:repelling-limit-REPP-2d-special2} and, most of all, the formulas to compute their distributions provide a new decisive step towards the full understanding of such processes. 

The convergence results in  \cite{M77} assume $\alpha$-mixing, while, in both \cite{H87,N02},  the stationary stochastic processes are assumed to satisfy a condition $\Delta$, similar to Leadbetter's $D$ condition (see \cite{LLR83}). None of these assumptions is easily verified for the stochastic processes  \eqref{eq:SP-def} arising from the dynamics that we consider here. For this reason, we needed to prove the convergence results under some weaker assumption that we call $\D_q^*$ so that we can apply  the results to dynamical systems.  Hence, we believe that the results here are also of interest in the classical field of extremes because we prove the complete convergence of general stationary stochastic processes under weaker assumptions on the long range dependence structure.

The techniques we use to prove the complete convergence build up on an idea introduced in \cite{FFT12} and further elaborated in \cite{FFT13,FFT15}, which consists in realising that, in the limit, the probability of having no exceedances can be approximated by the probability of having no escapes (a terminology introduced in \cite{FFT12}). Usually, in the dynamical setting, an exceedance corresponds to an entrance in a certain ball around the point $\zeta\in\X$, where $\varphi$ is maximised, while an escape corresponds to an entrance in a certain annulus around $\zeta$. This idea stated in Proposition~\ref{prop:relation-balls-annuli-general} below is at the core of the argument that allows us to prove the complete convergence under weaker assumptions, which can be easily verified by the systems we consider, and to obtain the new formulas to describe the limiting process, such as in \eqref{eq:limit-Anq}.

As a consequence of the complete convergence of the empirical multi-dimensional REPP,  we obtain the convergence of extremal processes, record times and record values point processes and one-dimensional REPP for stochastic processes as in \eqref{eq:SP-def}, both in the presence and absence of clustering. These corollaries follow by application of the continuous mapping theorem. In particular, this allows us to recover the main result in \cite{FFT13}, which proved the convergence of one-dimensional REPP to a compound Poisson process with Poisson events charged with a geometric multiplicity distribution, in the presence of clustering caused by a repelling periodic point $\zeta$. The geometric distribution is explained by the fact that a long cluster appears whenever the orbit enters a preimage of high order of the ball and the backward contraction rate at $\zeta$ gives the exponential factor. This can now be fully appreciated graphically by the distribution of the points on the vertical piles (see Figure~\ref{fig:processes}). 

The first results about extremal processes and records in the dynamical systems context were proved in \cite{HT17}, in the absence of clustering. Hence, another novelty here is study of the effect of clustering  in the convergence of extremal processes, record time and record value point processes for  dynamically generated processes, as in \eqref{eq:SP-def}. In fact, even in the classical context of general stationary stochastic processes satisfying the $\Delta$ condition, we could not find, in the literature, a proof of the continuity of the projection to $\mathbb D((0,\infty))$,  the space of functions on $(0,\infty)$ that are right continuous and have left-hand limits, equipped with the Skorokhod's $M_1$ topology, which we also provide to study the convergence of extremal processes. We remark that the proofs of continuity of the projection to $\mathbb D((0,\infty))$, equipped with the Skorokhod's $J_1$ topology used in the iid case are not sufficient for our applications since the $J_1$ topology is not adequate to deal with the stacking of points on the vertical direction. The Skorokhod's $M_1$ topology was also used recently in \cite{MZ15} to prove convergence to stable L\'evy processes. 

Regarding the study of record time and record value point processes, we are able to establish their convergence when the point $\zeta$ is a repelling periodic point, in which case the clustering has no effect on the limiting process obtained. However, we give an example where the convergence cannot be established at all because of the particular way how points get stacked on the same vertical pile. This contrasts with the situation where there is no clustering and, to our knowledge, is another novelty.

%
%
%The same comments apply to the convergence of extremal processes considered in \cite{R75, HT17}.

Among the dynamical systems to which we can apply the complete convergence results, we mention uniformly expanding systems, such as Rychlik maps \cite{R88} or higher dimensional versions such as in \cite{S00}, non-uniformly expanding systems which admit a uniformly expanding first return time induced system, such as Manneville-Pomeau \cite{PM80} or Liverani-Saussol-Vaienti maps \cite{LSV99} or Misiurewicz quadratic maps \cite{M81}. %and uniformly hyperbolic systems such as the toral automorphisms considered in \cite{CFFH15}. 
Moreover, in the absence of clustering, our results can also be applied to H\'enon maps, billiards with exponential and polynomial decay of correlations and general non-uniformly hyperbolic systems admitting Young towers with exponential and polynomial tails.

%\section{The setting}

\section{Complete convergence of empirical Rare Events Point Processes}

\subsection{Review of point process theory}

The theory of point processes provides a very powerful tool to study the extremal behaviour of stochastic processes. We start by establishing some notation, the definition of point processes and recalling some useful facts. For more details on the subject  see \cite{K86} and \cite[Chapter 3]{R87}. 

\subsubsection{Definition of point processes}

Let $E$ be a locally compact topological space with countable basis and $\mathscr E$ its Borel $\sigma$-algebra. In the applications below $E$ is a subset of $\R^d$ for some $d\in\N$. 

A \emph{point measure} on  $E$ is a measure $m$ of the following form:
$$
m=\sum_{i=1}^\infty \delta_{x_i},\quad \mbox{where  $\delta_{x_i}$ is the dirac measure at $x_i\in E$.}
$$
The measure $m$ is said to be \emph{Radon} if $m(K)<\infty$, for all compact sets $K\in\mathscr E$, and \emph{simple} if all the $x_i$, with $i=1, 2,\ldots$, are distinct. Let $M_p(E)$ be the space of all Radon point measures on $E$ and $\mathscr M_p(E)$ be the smallest $\sigma$-algebra containing the sets $\{m\in M_p(E): m(F)\in B\}$ for all $F\in\mathscr E$ and all Borel $B\subset[0,\infty)$. We endow $M_p(E)$ with the vague topology so that it becomes a complete, separable metric space. A \emph{point process} on E is a measurable map from a probability space to $M_p(E)$ equipped with its $\sigma$-algebra $\mathscr M_p(E)$. For example, $N_n$ described by \eqref{EREPP-def1} is a point process on $E=[0,\infty)^2$, since it can be formally written as a measurable map from the probability space $(\X,\mathcal B, \mu)$ to $(M_p(E),\mathscr M_p(E))$.

\subsubsection{Weak convergence of point processes} Consider a sequence of point processes $(N_n)_{n\in\N_0}$ defined on some probability space $(\Omega, \mathcal F, \p)$ and taking values on $M_p(E)$ equipped with $\mathscr M_p(E)$. We say that $N_n$ converges weakly to $N_0$, which we denote by $N_n\Rightarrow N_0$, if for every bounded, continuous real valued function $f$ on $M_p(E)$ (endowed with the vague topology), we have:
$$
\int_{M_p(E)} f \,d(N_n)_*\p\rightarrow \int_{M_p(E)} f \,d(N_0)_*\p, \quad \mbox{as $n\to\infty$,}
$$ 
where $(N_n)_*\p$ is the measure defined on $\mathscr M_p(E)$ corresponding to the pushforward of $\p$ by $N_n$, for each $n\in\N_0$. Proving weak convergence using the definition is often quite hard and, in fact, we will use the following criteria due to Kallenberg \cite{K86}, which applies when the limiting point process is simple. 
\begin{theorem}[Kallenberg's criteria]
\label{thm:Kallenberg}
Let $\mathscr S$ be a semi-ring that generates the ring $\mathscr R$, which, in turn, generates $\mathscr E$. Assume that $N_0$ is a simple point process on $M_p(E)$ and $\p(N_0(\partial F)=0)=1$, for all $F\in\mathscr R$. If we have
\begin{equation*}
\mbox{(A)}\quad \lim_{n\to\infty} \p(N_n(F)=0)=\p(N_0(F)=0)\qquad \mbox{(B)} \quad \lim_{n\to\infty} \E(N_n(G))=\E(N_0(G))<\infty,
\end{equation*}
 for all $F\in\mathscr R$ and $G\in\mathscr S$, then $N_n\Rightarrow N_0$.
\end{theorem}

\begin{remark}\label{rem:equality-pp}
Observe that two simple point processes $N$ and $N^*$ are identical in distribution if for all $F\in\mathscr R$ we have
$$\p(N(F)=0)=\p(N^*(F)=0).$$
See \cite[Proposition~3.23]{R87} or \cite[Theorem~3.3]{K86}.
\end{remark}

The power of weak convergence of point processes comes from the fact that from a weak convergence result one can obtain several corollaries simply by using the Continuous Mapping Theorem (CMT). In fact, the CMT applied in this setting, allows one to say that if $N_n\Rightarrow N_0$ and $h$ is a map from $M_p(E)$, endowed with the vague topology, to some other metric space, such that the probability of $N_0$ landing on the set of the discontinuities of $h$ is equal to $0$, then  $h(N_n)\Rightarrow h(N_0)$.

\subsection{Empirical and limiting point processes} We start by defining the empirical one-dimensional and two-dimensional REPP. We remark that our main goal is to study them in the dynamical context but since our results actually apply in a much wider framework, we only assume for now that $X_0, X_1, \ldots$ is a stationary stochastic process defined on the probability space $(\Omega, \mathcal F, \p)$, not necessarily arising from a dynamics as in \eqref{eq:SP-def}. 

In order to give a precise definition of the empirical REPP, we need a normalising sequence of thresholds $(u_n(\tau))_{\n\in\N}$, such that the average number of exceedances of $u_n(\tau)$ among the first $n$ r.v. of the process is asymptotically constant and equal to $\tau>0$, which can be seen as the asymptotic frequency of exceedances. Namely, we require that
\begin{equation}
\label{un}
n\p(X_0>u_n(\tau))\rightarrow \tau> 0,\quad \mbox{as $n\to\infty$.} 
\end{equation}

Of course the higher the frequency $\tau$ the lower the corresponding thresholds $u_n(\tau)$ should be. In fact, we assume that for each $n\in\N$, the level function $u_n(\tau)$ is continuous and strictly decreasing in $\tau$. Hence, $u_n$ has an inverse function, $u_n^{-1}$, which, for each value $z$ on the range of the r.v. $X_0$, returns the asymptotic frequency $\tau=u_n^{-1}(z)$ that corresponds to the average number of exceedances, among $n$ observations, of a threshold placed at the value $z$.
\subsubsection{One-dimensional REPP}
\begin{definition}
\label{def:1d-REPP}
We define the empirical one-dimensional REPP as
$$\displaystyle
N_n^1=\sum_{j=0}^{\infty} \delta_{j/n}\I_{X_j>u_n}\;.
$$
\end{definition}
In the classical setting of stationary stochastic processes, in the absence of clustering, in \cite{L76} these processes are shown to converge to a homogenous Poisson process of intensity $\tau$. 
In \cite{HHL88}, the authors consider the presence of clustering case and show that the limiting process is a compound Poisson process, which can be written as 
\begin{equation}
\label{eq:CPP}
\tilde N^1=\sum_{j=1}^\infty D_j\delta_{T_j},
\end{equation} where  $T_j=\sum_{\ell=0}^j \bar T_\ell$ is based on an iid sequence $(\bar T_j)_{j\in\N}$ such that $\bar T_j\stackrel[]{D}{\sim} \text{Exp}(\theta\tau)$, $j=1,2 \ldots$, and $(D_j)_{j\in\N}$ is an iid sequence of positive integer r.v. independent of the sequence $(T_j)_{j\in\N}$. 
The idea is that, in the limit, the exceedances in the same cluster get concentrated in the same Poisson time event that has now a multiplicity, 
which corresponds to the cluster size, whose average is $1/\theta$. The parameter $0<\theta<1$ is called the \emph{Extremal Index} (EI) and gives an indication of the intensity of clustering. In fact, in most situations the EI can be seen as the inverse of the average cluster size.

In \cite{FFT10}, the convergence to a homogenous Poisson process, $N^1$, was proved for processes arising from dynamical systems where $\zeta$ was a typical point w.r.t. $\mu$. In \cite{FFT13}, the results of \cite{HHL88} were generalised and applied to stochastic processes arising from dynamical systems, where the observable $\varphi$ was maximised at a repelling periodic point $\zeta$. The parameter $\theta$ depended on the expansion rate at $\zeta$, which, for 1-dimensional systems with a measure $\mu$ absolutely continuous w.r.t. Lebesgue and sufficiently regular, can be written as $\theta=1-1/|DT_\zeta^p|$, where $DT_\zeta^p$ is the derivative of $T^p$ at the periodic point $\zeta$ of period $p$. Moreover, a new formula to compute the multiplicity distribution was provided and, in these cases of repelling periodic points, it was shown the appearance of  a geometric multiplicity distribution, namely, $\pi(\kappa)=\p(D_i=\kappa)
=\theta(1-\theta)^{\kappa-1}$. This compound Poisson distribution was also observed in \cite{HV09}. Then, in \cite{AFV15} a dichotomy was established: either $\zeta$ is not periodic and the empirical REPP converges to a Poisson process or $\zeta$ is periodic  and, in this case, a compound Poisson process with geometric multiplicity distribution applies. Different multiplicity distributions were obtained when $\zeta$ was a discontinuity point of the system (\cite{AFV15}) and when multiple correlated maximal points were considered in \cite{AFFR16}.
\subsubsection{Two-dimensional REPP}\label{subsubsec:2d-REPP}
\begin{definition}
\label{def:REPP-2d}
We define the empirical two-dimensional REPP as
$$\displaystyle
N_n=\sum_{j=0}^{\infty} \delta_{(j/n, u_n^{-1}(X_j))}\;.
$$
\end{definition}
This type of two-dimensional REPP, with a different normalisation on the second component, was studied earlier in \cite{P71,R75}, on the iid setting. In \cite{A78} and \cite[Chapter~5.3]{LLR83}, for stationary sequences satisfying a uniform mixing condition of Leadbetter's type and in the absence of clustering, the process $N_n$ is shown to converge weakly to a two-dimensional Poisson process with two-dimensional Lebesgue intensity measure, $N$, which can be described as follows.  Let $(\bar T_{i,j})_{i,j\in\N}$ be a matrix of iid r.v. with common $\text{Exp}(1)$ distribution and consider $(T_{i,j})_{i,j\in\N}$ given by: 
\begin{equation}
\label{T-matrix}
T_{i,j}=\sum_{\ell=1}^j \bar T_{i,\ell}.
\end{equation}
Note that the rows of $(T_{i,j})_{i,j\in\N}$ are independent. Let $(U_{i,j})_{i,j\in\N}$ be a matrix of independent r.v. such that, for all $j\in\N$, the r.v. $U_{i,j}\stackrel[]{D}{\sim} \mathcal U_{(i-1,i]}$, \ie $U_{i,j}$ has a uniform distribution on the interval $(i-1,i]$. We assume that the matrices $(T_{i,j})_{i,j\in\N}$ and $(U_{i,j})_{i,j\in\N}$ are mutually independent. Then we can write:
\begin{equation}
\label{Poisson-2d}
N=\sum_{i,j=1}^\infty \delta_{(T_{i,j}, U_{i,j})}.
\end{equation}
In the absence of clustering, the weak convergence $N_n\Rightarrow N$ was also established by Holland and Todd in \cite{HT17}, under weaker mixing assumptions (the authors introduce a condition $\mathcal D_r(u_n^{(k)})$ similar to $D_3(u_n)$ used in \cite{FFT10}), which allowed them to apply their results to stochastic processes arising from dynamical systems like the ones we consider here.

In the presence of clustering, in the dynamical setting, there are no results regarding the convergence of $N_n$. In the classical setting of stationary stochastic processes, we mention the results by Mori \cite{M77} for $\alpha$-mixing processes, by Hsing \cite{H87} for processes satisfying a condition $\Delta$, which is closely related to $D(u_n)$ introduced by Leadbetter, by Novak \cite{N02},  also under condition $\Delta$, and by Resnick and Zeber \cite{RZ13} for Markov chains. The Mori-Hsing characterisation tells us that the limiting process, $N$, admits the representation:
\begin{equation}
\label{Mori-Hsing}
\tilde N=\sum_{i,j=1}^\infty\sum_{\ell=1}^{K_{i,j}} \delta_{(T_{i,j}, U_{i,j}Y_{i,j,\ell})},
\end{equation}
where  $T_{i,j}$ is given by \eqref{T-matrix}, for $\bar T_{i,j}\stackrel[]{D}{\sim}\text{Exp}(\theta)$, $U_{i,j}\stackrel[]{D}{\sim} \mathcal U_{(i-1,i]}$ and $(Y_{i,j,\ell})_{\ell\in\N}$ are the points of a point process $\gamma_{i,j}$ on $[1,\infty)$, with 1 as an atom. All the $\gamma_{i,j}$, $i,j\in\N$ are identically distributed and $(\bar T_{i,j})_{i,j\in\N}$, $(U_{i,j})_{i,j\in\N}$ and $(Y_{i,j,\ell})_{i,j\in\N}$ are mutually independent for all $\ell\in\N$.

One of the drawbacks of the Mori-Hsing representation is the fact that little is known about the vertical points given from the point processes $\gamma_{i,j}$ and, in particular, how to compute their distributions. Novak rewrites the limiting process using one-dimensional processes and an adaptation of the thinning technique described in \cite[Chapter~5.5]{LLR83}, but still no much insight regarding  the distribution of the piling on vertical direction is given. The results below fill this gap. 
\subsubsection{Two-dimensional REPP in the dynamical setting} 
\label{subsubsec:2d-DS}
In order to illustrate the potential of our approach, we consider stochastic processes arising from dynamical systems in the same setting described in \cite{FFT10}, \cite[Section~3.1]{FFT12} or \cite[Chapter~4.2.1]{LFFF16} and give a precise explicit formula for the limiting process. 
We recall that in most situations the observable $\varphi$ achieves a maximum at a point $\zeta$, which means that a natural condition then is to assume that, at least locally, $\varphi$ behaves as a decreasing function of the distance to $\zeta$ (for a given metric adequate to the study case). Hence, we assume that the observable $\varphi:\X\to\R\cup\{+\infty\}$ is of
the form
\begin{equation}
\label{eq:observable-form} \varphi(x)=g\left(\dist(x,\zeta)\right),
\end{equation} where $\zeta$ is a chosen point in the
phase space $\X$ and the function $g:[0,+\infty)\rightarrow {\mathbb
R\cup\{+\infty\}}$ is such that $0$ is a global maximum ($g(0)$ may
be $+\infty$); $g$ is a strictly decreasing bijection $g:V \to W$
in a neighbourhood $V$ of
$0$; and has one of the three types of behaviour described in \cite[Section~1.1]{FFT10} or \cite[Chapter~4.2.1]{LFFF16}. From the definition of $\varphi$, we have 
\begin{equation}
\label{eq:ball-link}
\{X_0>u_n\}=B_{g^{-1}(u_n)}(\zeta),
\end{equation} where $B_{\eps}(\zeta)$ denotes a ball of radius $\eps$ around $\zeta$. As observed in \cite[Remark~1]{FFT10}, the type of $g$ determines the tail of the distribution of $X_0$ and so there exists a clear link between the three types of $g$ and the three types of domains of attraction for maxima for the distribution of $X_0$.

Using \eqref{eq:observable-form} and \eqref{un}, in a similar way to the computation leading to \eqref{eq:ball-link}, we obtain that $u_n^{-1}(z)\sim n\mu\left(B_{g^{-1}(z)}(\zeta)\right)$, where we use the notation $A(n)\sim B(n)$, when $\lim_{n\to \infty} \frac{A(n)}{B(n)}=1$. Hence, in this situation, we can rewrite the two-dimensional REPP given in Definition~\ref{def:REPP-2d}, as measurable map from $(\X,\mathcal B, \mu)$ to $M_p([0,\infty)^2)$, in the following way:
\begin{equation}
\label{eq:REEP-2d-dynamic}
N_n(x)=\sum_{j=0}^\infty \delta_{\left(j/n, \,n\mu\left(B_{\dist(T^j(x),\zeta)}(\zeta)\right)\right)}\;.
\end{equation}
This formula helps to predict the liming process. Note that there is a contraction in the horizontal direction, which is responsible for the collapsing of the exceedances of the same cluster on the same vertical line, and an expansion, by the inverse factor, in the vertical direction, which makes it clear that only when the orbit is very close to $\zeta$, one has an exceedance of an high threshold, corresponding to a small limiting frequency $\tau$. 

In order to be more concrete, assume that $\mu$ is absolutely continuous with respect to Lebesgue measure with a sufficiently regular density so that $\mu(B_\eps(\zeta))\sim C\eps^d$, as $\eps\to0$, for some constant $C>0$ and where $d$ denotes the integer dimension of $\X$. Assume further that, as in \cite{FFT12,FFT13}, $\zeta$ is a repelling periodic point of prime period $p$ and, moreover, the derivative $DT_\zeta^p$ expands uniformly in every direction at a rate $\alpha$. This means that if $T^j(x)$ is very close to $\zeta$ then $\dist(T^{j+p}(x),\zeta)\sim \alpha\, \dist(T^j(x),\zeta)$. As a consequence, we will show that, in these cases, the limiting process can be written as:
\begin{equation}
\label{eq:repelling-limit-REPP-2d}
\tilde N=\sum_{i,j=1}^\infty\sum_{\ell=0}^{\infty} \delta_{(T_{i,j},\,\alpha^{\ell d}\, U_{i,j})},
\end{equation}
where  $T_{i,j}$ is given by \eqref{T-matrix}, for $\bar T_{i,j}\stackrel[]{D}{\sim}\text{Exp}(\theta)$, with $\theta=1-\alpha^{-d}$, $U_{i,j}\stackrel[]{D}{\sim} \mathcal U_{(i-1,i]}$ and, as before, $(\bar T_{i,j})_{i,j\in\N}$ and $(U_{i,j})_{i,j\in\N}$ are mutually independent. 

\begin{figure}[h]
\includegraphics[height=4.5cm]{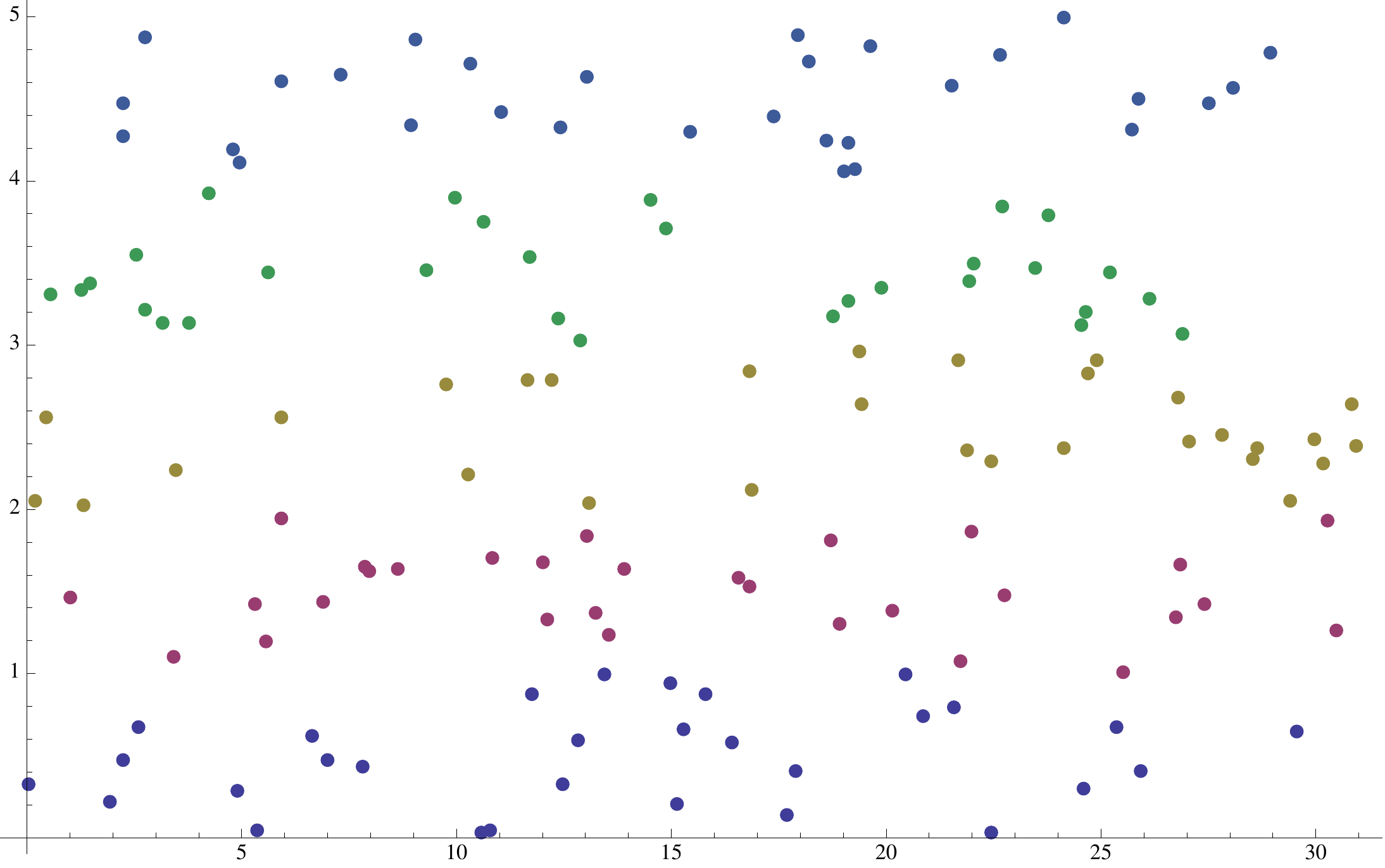}\qquad \includegraphics[height=4.5cm]{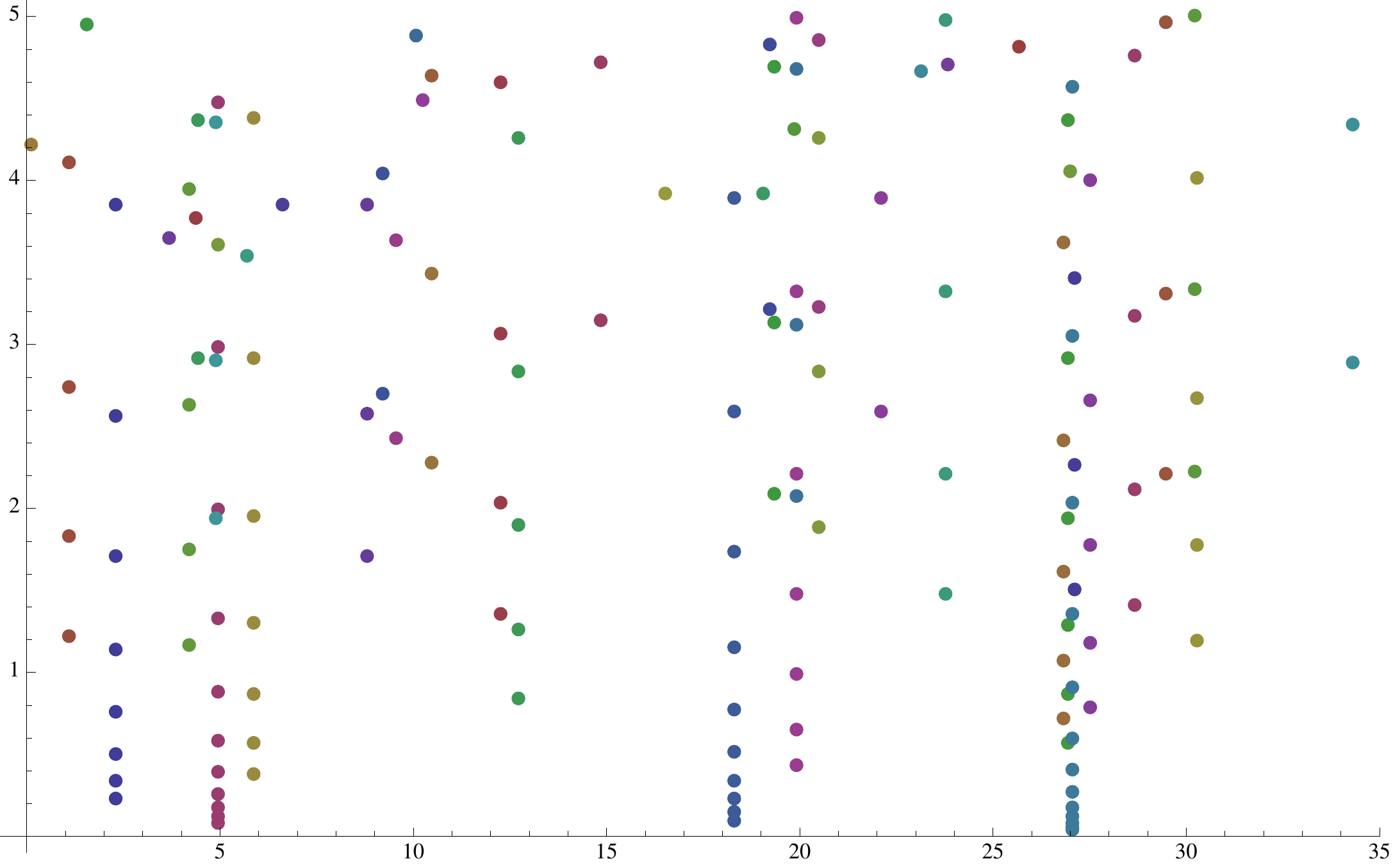}
\caption{Simulation of the two-dimensional Poisson process $N$ given by \eqref{Poisson-2d}, on the left, and simulation of the two-dimensional process $\tilde N$ given by \eqref{eq:repelling-limit-REPP-2d}, with $d=1$ and $\alpha=3/2$, on the right.}
\label{fig:processes}
\end{figure}

Note that this new formula completely describes the limiting two-dimensional process, including  the distribution of the points on vertical direction, and arises naturally from the dynamical system. In fact, for example, for one-dimensional uniformly expanding systems as considered by Rychlik, in \cite{R88}, a dichotomy holds. Either $\zeta$ is non-periodic and the two-dimensional REPP converges to the Poisson process given by \eqref{Poisson-2d}, or $\zeta$ is periodic and the limiting process is given by \eqref{eq:repelling-limit-REPP-2d}. Moreover, note that the geometric multiplicity distribution obtained in \cite{FFT13} for the limit of the one-dimensional REPP can be easily understood now from the geometric displacement of the points on the vertical lines, by the factor $\alpha^d$.

\subsubsection{Multi-dimensional REPP in the dynamical setting}\label{subsubsec:multi-dimensional}
For higher dimensional systems, when the expansion rate is not uniform in all directions, for example, it may be useful to keep track of the movement of the points in all directions so that we can understand completely how does the distance to $\zeta$ evolve, which is directly tied to the magnitude of the observations (see Corollary~\ref{cor:two-dimensional}). %and not only the distance to $\zeta$. 
In order to do this, we consider the tangent space $T_\zeta\X$ and, assuming $\zeta$ is an hyperbolic periodic point (in our case a repelling periodic point), we let $\Phi_\zeta:V\to W$ denote a diffeomorphism, defined on an open ball, $V$, around $\zeta$ in $T_\zeta\X$ onto a neighbourhood, $W$, of $\zeta$ in $\X$, such that $\Phi_\zeta(E^{s,u}\cap V)=W^{s,u}(\zeta)\cap W$.

In what follows we will use the notation $|\cdot|$  and $\leb(\cdot)$ for Lebesgue measure, \ie for every borelian set $A$, we write $|A|=\leb(A)$ for the Lebesgue measure of $A$.
\begin{definition}
\label{def:multi-dimensional-REPP}
We define an empirical multi-dimensional REPP process $N_n^*:\X\to M_p([0,\infty)\times T_\zeta\X)$ by $$\displaystyle
N_n^*=\sum_{j=0}^{\infty} \delta_{\left(j/n, \frac{\Phi_\zeta^{-1}(T^j(x))}{g^{-1}(u_n(1))|B_1(\zeta)|^{1/d}}\right)}\;.
$$
\end{definition}
In the absence of clustering, like when $\zeta$ is not periodic, the limiting process for $N_n^*$ is a multi-dimensional homogeneous Poisson process with intensity measure given by Lebesgue measure, which can be written as:
\begin{equation}
\label{eq:repelling-limit-REPP-multi-d-no-clustering}
N^*=\sum_{i,j=1}^\infty\delta_{(T_{i,j},\, U_{i,j})},
\end{equation}
where  $T_{i,j}$ is given by \eqref{T-matrix}, for $\bar T_{i,j}\stackrel[]{D}{\sim}\text{Exp}\left(\leb\left(B_{i}(\zeta)\setminus B_{i-1}(\zeta)\right)\right)$, with $B_{i}(\zeta)$ denoting the ball of radius $i\in\N$ around $\zeta$ on $T_\zeta\X$, $B_{0}(\zeta)=\emptyset$; $U_{i,j}\stackrel[]{D}{\sim} \mathcal U_{B_{i}(\zeta)\setminus B_{i-1}(\zeta)}$, which means that $U_{i,j}$ is a point of  $B_{i}(\zeta)\setminus B_{i-1}(\zeta)$ selected randomly according to normalised $d$-dimensional Lebesgue measure on $T_\zeta\X$, \ie $\leb/(\leb(B_{i}(\zeta)\setminus B_{i-1}(\zeta)))$.  As before, $(\bar T_{i,j})_{i,j\in\N}$ and $(U_{i,j})_{i,j\in\N}$ are mutually independent.

In the presence of clustering, we will show that, under the same assumptions on the regularity of $\mu$ with respect to Lebesgue measure, as before, then there exists some $0<\theta<1$ such that the limiting process can be written as:
\begin{equation}
\label{eq:repelling-limit-REPP-multi-d}
N_\theta^*=\sum_{i,j=1}^\infty\sum_{\ell=0}^{\infty} \delta_{(T_{i,j},\,DT_\zeta^{\ell p}( U_{i,j}))},
\end{equation}
where the variables  $T_{i,j}$ and $U_{i,j}$ are as described above, except for $\bar T_{i,j}$, which is now such that $\bar T_{i,j}\stackrel[]{D}{\sim}\text{Exp}\left(\theta\leb\left(B_{i}(\zeta)\setminus B_{i-1}(\zeta)\right)\right)$, instead. In the case of $\zeta$ being a repelling periodic point, we will see that $\theta=1-|\det^{-1}(DT^p_\zeta)|$.

\section{General complete convergence of multi-dimensional empirical REPP}
The main goal is to study time series arising from discrete dynamical systems as described in \eqref{eq:SP-def}. Yet, as the results we need to develop apply to general stationary stochastic processes, in this section, we identify $X_0, X_1, \ldots$  with the respective coordinate-variable process on $(\mathcal S^{\N_0}, \mathcal B^{\N_0}, \p)$, given by Kolmogorv's existence theorem, where $\mathcal S=\R^d$ and  $\mathcal B^{\N_0}$ is the $\sigma$-field generated by the coordinate functions $Z_n:\mathcal S^{\N_0}\to\mathcal S$, with $Z_n(x_0,x_1,\ldots)=x_n$, for $n\in\N_0$, so that there is a natural measurable map, the shift operator $\TF: \mathcal S^{\N_0}\to\mathcal S^{\N_0}$, given by $\TF(x_0,x_1,\ldots)=(x_1,x_2,\ldots)$, which when applied later in the dynamical systems context can be identified with the action of time on the system. Except for the multi-dimensional processes considered in Section~\ref{subsubsec:multi-dimensional}, we always consider the case $d=1$.

Note that, under these identifications, we can write:
\[
X_{i-1}\circ \TF =X_{i}, \quad \mbox{for all $i\in\N$}.
\]
Since, we assume that the process is stationary, then $\p$ is $\TF$-invariant. Note that $X_i=X_0\circ \TF^i$, for all $i\in\N_0$, where $\TF^i$ denotes the $i$-fold composition of $\TF$, with the convention that $\TF^0$ denotes the identity map on $\R^\N$. 

In what follows, for every $A\in\mathcal B$, we denote the complement of $A$ as $A^c:=\mathcal X\setminus A$.

Let $A\in\mathcal B$ be an event and let $J$ be an interval contained in $[0,\infty)$. We define
\begin{equation}
\label{eq:W-def}
\mathscr W_{J}(A)=\bigcap_{i\in J\cap \N_0}\TF^{-i}(A^c).
\end{equation}

We will write $\mathscr W_{J}^c(A):=(\mathscr W_{J}(A))^c$.

Consider the event $A\in\mathcal B$ and define, for some $j\in \N$,  
\begin{equation} 
\label{eq:A^j}
A^{(j)}:=A\cap \TF^{-1}(A^c)\cap\ldots\cap\TF^{-j}(A^c),
\end{equation} and, for $j=0$, we simply define $A^{(0)}=A$.

Let $(A_n)_{n\in\N}$ be a sequence of events in $\mathcal B$. For each $n\in\N$, let $R_{n}^{(j)}=\min\{r\in\N: A_n^{(j)}\cap \TF^{-r}A_{n}^{(j)}\neq\emptyset\}$.
We assume that there exists $q\in\N_0$ such that:
\begin{equation}
\label{def:q}
q=\min\left\{j\in\N_0:\lim_{n\to\infty} R_{n}^{(j)}=\infty\right\}.
\end{equation}

\subsection{Dependence conditions}
\label{subsec:dependence}
In what follows, for some $a>0$, $y\in\R^d$ and a set $A\subset \R^d$, we define the sets $aA\subset \R^d$ and $A+y$, as $aA=\{ax:\, x\in A\}$ and $A+y=\{x+y:\, x\in A\}$, respectivley.
We introduce a mixing condition which is specially designed for the application to the dynamical setting. 

Let $E=\cup_{k=1}^m E_k$, where for each $k=1, \ldots,m\in\N$,
\begin{equation}
\label{eq:rectangles}
E_k=J_k\times A_k,\; J_k=[a_k,b_k),\; \displaystyle A_k=\bigcup_{s=1}^{\varsigma_k} G_{k,s},\; \mbox{where}
\end{equation} 
\begin{equation}
\label{eq:rectangles-cases}
G_{k,s}=(\tau_{k,2s-1}, \tau_{k,2s}]\;\quad\mbox{or}\;\quad G_{k,s}=(e_{k,s,1}, f_{k,s,1}]\times\ldots\times(e_{k,s,d}, f_{k,s,d}],
\end{equation}
for $0\leq a_1<b_1\leq a_2<b_2\leq\ldots\leq a_k<b_k$ and $G_{k,i}\cap G_{k,j}=\emptyset$ for $i\neq j=1, \ldots,\varsigma_k$, which in the first case means that  $0\leq \tau_{k,1}< \tau_{k,2}<\ldots<\tau_{k,2\varsigma_k-1}< \tau_{k,2\varsigma_k}$.% and in the second case, 
%$e_{k,1,i}< f_{k,1,i}<\ldots<e_{k,\varsigma_k,i}< f_{k,\varsigma_k,i}\in\R$, for all $i=1,\ldots,d$.  

To study the two-dimensional empirical REPP introduced in Section~\ref{subsubsec:2d-REPP}, for each  $n\in\N$, we define 
\begin{equation}
\label{eq:rectangles-n}
J_{n,k}=nJ_k,\; G_{n,k,s}=\{u_n(\tau_{k,2s})<X_0\leq u_n(\tau_{k,2s-1})\},\; A_{n,k}=\bigcup_{s=1}^{\varsigma_k} G_{n,k,s}\;\mbox{and}\; E_{n,k}=J_{n,k}\times A_{n,k},
\end{equation}
where $u_n(\tau)$ is defined as in \eqref{un}.

To study the multi-dimensional empirical REPP introduced in Section~\ref{subsubsec:multi-dimensional}, for each  $n\in\N$, we define 
\begin{equation}
\label{eq:rectangles-n-bi}
J_{n,k}=nJ_k,\;  A_{n,k}=\Phi_\zeta(g^{-1}(u_n(1))|B_1(\zeta)|^{1/d}A_{k})\;\mbox{and}\; E_{n,k}=J_{n,k}\times A_{n,k},
\end{equation}
where $u_n(1)$ is defined as in \eqref{un}.

\begin{condition}[$\D_q^*(u_n)$]\label{cond:D*} We say that $\D_q^*(u_n)$ holds for the sequence $X_0,X_1,\ldots$ if for every  $m, t,n\in\N$ and every $J_k$ and $A_k$, with $k=1, \ldots, m$, chosen as in \eqref{eq:rectangles}, we have 
\begin{equation}\label{eq:D1}
\left|\p\left(\A_{n,k}\cap \bigcap_{i=k}^m\W_{J_{n,i}}\left(\A_{n,i}\right) \right)-\p\left(\A_{n,k}\right)
  \p\left( \bigcap_{i=k}^m\W_{J_{n,i}}\left(\A_{n,i}\right)\right)\right|\leq \gamma(q,n,t),
\end{equation}
where, for $k=1,\ldots,n$, each $J_{n,k}$ and $A_{n,k}$ is associated to $J_k$ and $A_k$, respectively, as in \eqref{eq:rectangles-n} or \eqref{eq:rectangles-n-bi}, $\A_{n,k}$ is associated to $A_{n,k}$ by \eqref{eq:A^j}, $\min\{J_{n,k}\cap\N_0\}\geq t$ and $\gamma(q,n,t)$ is decreasing in $t$ for each $n$ and there exists a sequence $(t_n)_{n\in\N}$ such that $t_n=o(n)$ and
$n\gamma(q,n,t_n)\to0$ when $n\rightarrow\infty$.
\end{condition}

For some fixed $q\in\N_0$, consider the sequence $(t_n)_{n\in\N}$, given by condition  $\D_q(u_n)^*$ and let $(k_n)_{n\in\N}$ be another sequence of integers such that 
\begin{equation}
\label{eq:kn-sequence}
k_n\to\infty\quad \mbox{and}\quad  k_n t_n = o(n).
\end{equation}

\begin{condition}[$\D'_q(u_n)$]\label{cond:D'q} We say that $\D'_q(u_n)$
holds for the sequence $X_0,X_1,X_2,\ldots$ if there exists a sequence $(k_n)_{n\in\N}$ satisfying \eqref{eq:kn-sequence} and such that,  for every $m$ and every $A_k$, with $k=1, \ldots, m$, chosen as in \eqref{eq:rectangles}, we have
\begin{equation}
\label{eq:D'rho-un}
\lim_{n\rightarrow\infty}\,n\sum_{j=1}^{\lfloor n/k_n\rfloor-1}\p\left( \A_{n,k}\cap \TF^{-j}\left(\A_{n,k}\right)
\right)=0,
\end{equation}
where, for $k=1,\ldots,n$, each $A_{n,k}$ is associated to $A_k$, as in \eqref{eq:rectangles-n} or \eqref{eq:rectangles-n-bi}, and $\A_{n,k}$ is associated to $A_{n,k}$ by \eqref{eq:A^j}
\end{condition}

\subsection{The stacking distribution}

We assume the existence of a $\sigma$-finite outer measure $\nu$ on the positive real line so that the following limits exist for every $k=1,\ldots,m$
\begin{equation}
\label{eq:limit-Anq}
\lim_{n\to\infty} n\p(\A_{n,k})=\nu(A_k),
\end{equation}
where $\A_{n,k}$  is given by equation \eqref{eq:rectangles-n} or \eqref{eq:rectangles-n-bi} and $A_k$ by \eqref{eq:rectangles}. The role played by the outer measure $\nu$ is crucial for determining the distribution of the mass points of the limiting process that get aligned at the same Poissonian time event and, ultimately, for determining the limiting process itself.

In the absence of clustering, \ie if $\D'_0(u_n)$ holds then $A_{n,k}^{(0)}=A_{n,k}$ and, by \eqref{un}, we have that $\lim_{n\to\infty}n\p(G_{n,k,s})=\lim_{n\to\infty}n\left(\p(X_0>u_n(\tau_{k,2s}))-\p(X_0>u_n(\tau_{k,2s-1}))\right)=|G_{k,s}|$. Therefore, in this case, we have $\nu=\leb$ and the 2-dimensional empirical REPP $N_n$ given in Definition~\ref{def:REPP-2d} converges to a 2-dimensional Poisson process given in \eqref{Poisson-2d}.

In the setting of Section~\ref{subsubsec:2d-DS}, when the limiting process is given by \eqref{eq:repelling-limit-REPP-2d}, we expect that 
\begin{equation}
\lim_{n\to\infty} n\p(A_{n,k})=\lim_{n\to\infty} n\p\left(\bigcup_{s=1}^{\varsigma_k}G_{n,k,s}\right)=\sum_{s=1}^{\varsigma_k}|G_{k,s}|=|A_k|
\end{equation}
where $G_{n,k,s}$ and $G_{k,s}$ are given by equations \eqref{eq:rectangles-n} and \eqref{eq:rectangles}, respectively. Moreover,
\begin{equation*}
\lim_{n\to\infty} n\p(\A_{n,k})=%\lim_{n\to\infty} n\p\left(\bigcup_{s=1}^{\varsigma_k}G_{n,k,s}^{(q)}\right)=
\sum_{s=1}^{\varsigma_k}\big|G_{k,s}\setminus\bigcup_{s=1}^{\varsigma_k}\bigcup_{j=1}^\infty(1-\theta)^j G_{k,s}\big|=\big|A_k\setminus\bigcup_{j=1}^\infty(1-\theta)^j A_{k}\big|,
\end{equation*}
where $\theta=1-\alpha^{-d}$.
%where $\alpha G_{k,s}=(\alpha\tau_{k,2s-1}, \alpha\tau_{k,2s}]$ and $\alpha A_k=\cup_{s=1}^{\varsigma_k}\alpha G_{k,s}$, for every $\alpha>0$. 
In this case $\nu$ is given by 
\begin{equation}
\label{eq:outer-measure-ex-1}
\nu(A)=\Bigg|A\setminus\bigcup_{j=1}^\infty(1-\theta)^jA\Bigg|.
\end{equation}

\begin{remark}
Observe that when $A_k=[0,\tau)$ then $A_k\setminus(1-\theta)A_k=[(1-\theta)\tau,\tau)$ and $|A_k\setminus(1-\theta)A_k|=\theta|A_k|=\theta\tau.$ Moreover, in this case where $A_k=[0,\tau)$, we can also write that $\theta=\lim_{n\to\infty}\theta_n$, where
\[\theta_n=\frac{\p(\A_{n,k})}{\p(A_{n,k})}.\]
\end{remark}

In the setting of Section~\ref{subsubsec:multi-dimensional}, using uniform magnification of Lebesgue measure, we have (see equation \eqref{eq:normalisation} below)
$
\lim_{n\to\infty} n\p(A_{n,k})=%\lim_{n\to\infty} nC|g^{-1}(u_n(1))A_k|=\lim_{n\to\infty} n\left(g^{-1}(u_n(1))\right)^dC|A_k|=c
|A_k|.
$ Furthermore,
$
\lim_{n\to\infty} n\p(\A_{n,k})=\big|A_k\setminus\bigcup_{j=1}^\infty DT_\zeta^{-j}( A_{k})\big|,
$
so that, in this case, $\nu$ is given by 
\begin{equation}
\label{eq:outer-measure-ex-2}
\nu(A)=\Bigg|A\setminus\bigcup_{j=1}^\infty DT_\zeta^{-j}(A)\Bigg|.
\end{equation}

\subsection{Complete convergence}\label{subsubsec:complete-convergence}
The cornerstone of the main convergence result in this section is the following proposition that essentially asserts that the non-occurrence of the asymptotically rare event $A_{n,k}$ during a conveniently normalised time frame can be replaced by the non-occurrence of the event $\A_{n,k}$ on the same normalised time frame, up to an asymptotically negligible error. This idea goes back to \cite[Proposition~1]{FFT12} and was further elaborated in \cite[Proposition 2.7]{FFT15}.

\begin{proposition}
\label{prop:relation-balls-annuli-general} Observe that
\[\left|\p\left(\bigcap_{k=1}^m\W_{J_{n,k}}\left(\A_{n,k}\right)\right)-\p\left(\bigcap_{k=1}^m\W_{J_{n,k}}\left(A_{n,k}\right)\right)\right|\leq q\sum_{k=1}^m\p(A_{n,k}).\]
\end{proposition}

We are now ready to state a general complete convergence result.

\begin{theorem}\label{thm:indep-increments} 
Let $m\in\N$ and for each $k=1,\ldots,m$ let $E_k,\; J_k,\;A_k$ be given, as in  \eqref{eq:rectangles}, and $G_{k,s}$, as in \eqref{eq:rectangles-cases}, with $s=1, \ldots, \varsigma_k$ . For $n\in\N$, consider the respective versions $E_{n,k},\; J_{n,k},\;A_{n,k}$ and $G_{n,k,s}$ given by \eqref{eq:rectangles-n} or \eqref{eq:rectangles-n-bi}. Let $q$ be as in \eqref{def:q} and assume that conditions $\D_q^*(u_n)$ and $\D'_q(u_n)$ hold. Also assume that there exists a $\sigma$-finite outer measure $\nu$ on $\R^+$ such that \eqref{eq:limit-Anq} holds. Hence,
\[
\lim_{n\to\infty}\p\left(\bigcap_{k=1}^m\W_{J_{n,k}}\left(A_{n,k}\right)\right)=\lim_{n\to\infty}\p\left(\bigcap_{k=1}^m\W_{J_{n.k}}\left(\A_{n,k}\right)\right)=\prod_{k=1}^m \e^{-\nu(A_k)|J_k|}.
\]
\end{theorem}

\begin{remark}
Note that using the right definition for $A_k$ and $A_{n,k}$, we have $\p(N_n(J_k\times A_k)=0)=\p\left(\bigcap_{k=1}^m\W_{J_{n,k}}\left(A_{n,k}\right)\right)$ and $\p(N^*_n(J_k\times A_k)=0)=\p\left(\bigcap_{k=1}^m\W_{J_{n,k}}\left(A_{n,k}\right)\right)$.
\end{remark}

Observe that the first equality in Theorem~\ref{thm:indep-increments} follows from the definition of $A_{n,k}$ and Proposition~\ref{prop:relation-balls-annuli-general}, whose proof we start with.  

\begin{proof}[Proof of Proposition~\ref{prop:relation-balls-annuli-general}]
Since $\A_{n,k}\subset A_{n,k}$, then clearly $\bigcap_{k=1}^m\W_{J_{n,k}}(A_{n,k})\subset\bigcap_{k=1}^m\W_{J_{n,k}}(\A_{n,k})$. Hence, we have to estimate the probability of $\bigcap_{k=1}^m\W_{J_{n,k}}(\A_{n,k})\setminus\bigcap_{k=1}^m\W_{J_{n,k}}(A_{n,k})$.

Let $x\in\bigcap_{k=1}^m\W_{J_{n,k}}(\A_{n,k})\setminus\bigcap_{k=1}^m\W_{J_{n,k}}(A_{n,k})$. Then, $\TF^i(x)\in A_{n,k}$ for some $i\in J_{n,k}=[n a_k,n b_k)$ and $k\in\{1,\ldots,m\}$. For that $k$, we will see that there exists $j\in\{1,\ldots,q\}$ such that $\TF^{n b_k-j}(x)\in A_{n,k}$. In fact, suppose that no such $j$ exists. Let $\ell=\max\{i\in J_{n,k}:\, \TF^i(x)\in A_{n,k}\}$. Then, clearly, $\ell<n b_k-q$. Hence, if $\TF^i(x)\notin A_{n,k}$, for all $i=\ell+1,\ldots,n b_k-1$, then we must have that $\TF^{\ell}(x)\in\A_{n,k}$. But this contradicts the fact that $x\in\W_{J_{n,k}}(\A_{n,k})$. Consequently, we have that there exists $j\in\{1,\ldots,q\}$ such that $\TF^{n b_k-j}(x)\in A_{n,k}$ and since $x\in\W_{J_{n,k}}(\A_{n,k})$ then we can actually write $\TF^{n b_k-j}(x)\in A_{n,k}\setminus\A_{n,k}$.

This means that $\bigcap_{k=1}^m\W_{J_{n,k}}(\A_{n,k})\setminus\bigcap_{k=1}^m\W_{J_{n,k}}(A_{n,k})\subset \bigcup_{k=1}^m\bigcup_{j=1}^q \TF^{-(n b_k-j)}(A_{n,k}\setminus\A_{n,k})$ and then
\begin{multline*}
\left|\p\left(\bigcap_{k=1}^m\W_{J_{n,k}}\left(\A_{n,k}\right)\right)-\p\left(\bigcap_{k=1}^m\W_{J_{n,k}}\left(A_{n,k}\right)\right)\right|=\p\left(\bigcap_{k=1}^m\W_{J_{n,k}}(\A_{n,k})\setminus\bigcap_{k=1}^m\W_{J_{n,k}}(A_{n,k})\right)\\
\leq\p\left(\bigcup_{k=1}^m\bigcup_{j=1}^q \TF^{-(n b_k-j)}\left(A_{n,k}\setminus\A_{n,k}\right)\right)\leq\sum_{k=1}^m\sum_{j=1}^q \p\left(A_{n,k}\setminus\A_{n,k}\right)\leq q\sum_{k=1}^m\p(A_{n,k}),
\end{multline*}
as required.
\end{proof}

Before going to the proof of Theorem~\ref{thm:indep-increments}, we state and prove some auxiliary lemmata, in which we will use the following notation $
\mathscr W_{s,\ell}(A):=\mathscr W_{[\lfloor s\rfloor,\lfloor s\rfloor+\max\{\lfloor\ell\rfloor-1,\ 0\}]}(A)
$
and  $\mathscr W_{s,\ell}^c(A):=(\mathscr W_{s,\ell}(A))^c$.
\begin{lemma}
\label{lem:time-gap-1}
For any fixed $A\in \mathcal B$ and $s,t',m\in\N$, %with $t< m$,
we have:
\begin{equation*}
\left|\p(\mathscr W_{0,s+t'+m}(A))-\p(\mathscr W_{0,s}(A)\cap \mathscr W_{s+t',m}(A))\right|\leq t'\p(A).
\end{equation*}
\end{lemma}

\begin{proof}
Using stationarity we have
\begin{align*}
\p(\mathscr W_{0,s}(A)\cap \mathscr W_{s+t',m}(A))-\p(\mathscr W_{0,s+t'+m}(A))&=\p(\mathscr W_{0,s}(A)\cap \mathscr W_{s,t'}^c(A)\cap\mathscr W_{s+t',m}(A))\\&\leq \p(\mathscr W_{0,t'}^c(A))=\p(\cup_{j=0}^{t'-1} \TF^{-j}(A))\\
& \leq \sum_{j=0}^{t'-1}\p(\TF^{-j}(A))=t'\p(A).
\end{align*}
\end{proof}

\begin{lemma}
\label{lem:inductive-step-1}For any fixed $A\in\mathcal B$ and integers $s,t,m$, we have:
\begin{multline*}
\left|\p(\mathscr W_{0,s}(A)\cap \mathscr W_{s+t,m}(A))-(1-s\p(A))\p(\mathscr W_{0,m}(A))\right|\leq\\
\left|s\p(A)\p(\mathscr W_{0,m}(A))-\sum_{j=0}^{s-1}\p(A\cap \mathscr W_{s+t-j,m}(A))\right|+s\sum_{j=1}^{s-1}\p(A\cap \TF^{-j}(A)).
\end{multline*}
\end{lemma}

\begin{proof}
Observe that the first term in the bound is measuring the mixing across the gap t and the second term is measuring the probability that two events $A$ appear in the first block. Adding and substracting and using the triangle inequality we obtain that
\begin{multline}
\label{eq:triangular1}
\big|\p(\mathscr W_{0,s}(A)\cap\mathscr W_{s+t,m}(A))-\p(\mathscr W_{0,m}(A))(1-s\p(A))\big|\leq \\
\left|s\p(A)\p(\mathscr W_{0,m}(A))-\sum_{j=0}^{s-1}\p(A\cap \mathscr W_{s+t-j,m}(A))\right| + \\
+\left|\p(\mathscr W_{0,s}(A)\cap\mathscr W_{s+t,m}(A))-\p(\mathscr W_{0,m}(A))+\sum_{j=0}^{s-1}\p(A\cap \mathscr W_{s+t-j,m}(A))\right|.
\end{multline}
Regarding the second term on the right, by stationarity, we have
\begin{align*}
\p(\mathscr W_{0,s}(A)\cap \mathscr W_{s+t,m}(A))&=\p(\mathscr W_{s+t,m}(A))-\p(\mathscr W_{0,s}^c(A)\cap \mathscr W_{s+t,m}(A))\\&=\p(\mathscr W_{0,m}(A))-\p(\mathscr W_{0,s}^c(A)\cap\mathscr W_{s+t,m}(A)).
\end{align*}
Now, since $\mathscr W_{0,s}^c(A)\cap\mathscr W_{s+t,m}(A)=\cup_{i=0}^{s-1}\TF^{-i}(A)\cap\mathscr W_{s+t,m}(A)$, we have by Bonferroni's inequality that
\begin{align*}
0\leq &\sum_{j=0}^{s-1}\p(A\cap \mathscr W_{s+t-j,m}(A))-\p(\mathscr W_{0,s}^c(A)\cap\mathscr W_{s+t,m}(A))\leq\\ 
&\sum_{j=0}^{s-1}\sum_{i>j}^{s-1}\p(\TF^{-j}(A)\cap \TF^{-i}(A)\cap\mathscr W_{s+t,m}(A)).
\end{align*}
Hence, using these last two computations we get: 
\begin{align}
\Big|\p(\mathscr W_{0,s}(A)\cap\mathscr W_{s+t,m}(A))&-\p(\mathscr W_{0,m}(A))+\sum_{j=0}^{s-1}\p(A\cap\mathscr W_{s+t-j,m}(A))\Big|\nonumber\\
&\leq\sum_{j=0}^{s-1}\sum_{i> j}^{s-1}\p(\TF^{-j}(A)\cap \TF^{-i}(A)\cap \mathscr W_{s+t,m}(A))\nonumber\\
&\leq\sum_{j=0}^{s-1}\sum_{i>j}^{s-1}\p(\TF^{-j}(A)\cap \TF^{-i}(A))\leq s\sum_{j=1}^{s-1}\p(A\cap \TF^{-j}(A)).\nonumber
\end{align}
The result now follows directly from plugging the last estimate into \eqref{eq:triangular1}.
\end{proof}

\begin{proof}[Proof of Theorem~\ref{thm:indep-increments}]
We split time into alternate blocks of different sizes as usual and follow particularly the notation of \cite[Proposition~1]{FFT10}. Let $h:=\inf_{k\in \{1,\ldots,m\}}\{b_k-a_k\}$ and
$H:=\lceil b_k \rceil$. Let $n$ be sufficiently
large so that $k_n>2/h$. Note this guarantees that if we partition $[0,Hn]\cap
{\mathbb Z}$ into blocks of length $r_n:=\lfloor n/k_n\rfloor$,
$I_1=[0,r_n)$, $I_2=[r_n,2r_n)$,\ldots,
$I_{Hk_n}=[(Hk_n-1)r_n,Hk_nr_n)$, $I_{Hk_n+1}=[Hk_nr_n,Hn)$, then there is more than one of these blocks contained in $ J_{n,i}$. Let $S_\ell=S_\ell(k_n)$ be the number of blocks $I_j$ contained in $J_{n,\ell}$, that is,
$$S_\ell:=\#\{j\in \{1,\ldots,Hk_n\}:I_j\subset J_{n,\ell}\}.$$
As we have already observed $S_\ell>1$ for all $\ell=1,\ldots,m$.
For each $\ell=1,\ldots,m$, we define
$$\mathscr A_\ell:=\bigcap_{k=\ell}^m\mathscr W_{J_{n,k}} \left(\A_{n,k}\right).$$
Set $i_\ell:=\min\{j\in\{1,\ldots,k_n\}:I_j\subset J_{n,\ell}\}.$
Then $I_{i_\ell},I_{i_\ell+1},\ldots,I_{i_\ell+S_\ell-1}\subset J_{n,\ell}$. Now, fix $\ell$ and for each $ i\in \{i_{\ell},\ldots,i_{\ell}+S_{\ell}-1\}$ let
$$
\mathscr B_{i,\ell}:=\bigcap_{j=i}^{i_{\ell}+S_{\ell}-1}\mathscr W_{I_j}(\A_{n,\ell}),\;  I_i^*:=[(i-1)r_n,
ir_n-t_n)\; \mbox{ and } I_i':=J_i-J_i^*.$$ Note that
$|I_i^*|=r_n-t_n=:r_n^*$ and $|I_i'|=t_n$. See Figure~\ref{fig:notation}
for more of an idea of the notation here.

\begin{figure}[h]
\includegraphics[height=7cm]{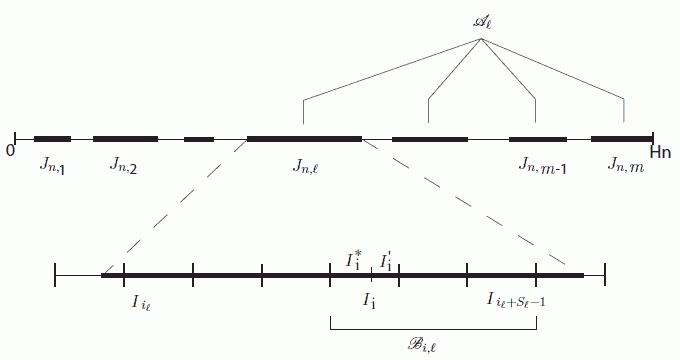}
\caption{Notation}
\label{fig:notation}
\end{figure}

Using Lemmas \ref{lem:time-gap-1} and \ref{lem:inductive-step-1}, we obtain
\begin{align*}
\Big|&\p(\BB_{i,\ell}\cap\AA_{\ell+1})-(1-r_n\p(\A_{n,\ell}))\p(\BB_{i+1,\ell}\cap\AA_{\ell+1})\Big|\leq\\
&\leq r_n^*\sum_{j=1}^{r_n^*-1}\p\left(\A_{n,\ell}\cap \TF^{-j}\A_{n,\ell}\right)+2t_n\p(\A_{n,\ell})\\
&\quad+ \sum_{j=0}^{r_n^*-1}\Big|\p(\A_{n,\ell})\p(\BB_{i+1,\ell}\cap\AA_{\ell+1})%\\&\qquad\qquad\quad
-\p\left(\TF^{-j-(i-1)r_n}\A_{n,\ell}\cap\BB_{i+1,\ell}\cap\AA_{\ell+1}\right)\Big|,
\end{align*}
Now using condition $\D_q^*(u_n)$, we obtain
%\begin{multline*}
\[\Big|\p(\BB_{i,\ell}\cap\AA_{\ell+1})-(1-r_n\p(\A_{n,\ell}))\p(\BB_{i+1,\ell}\cap\AA_{\ell+1})\Big|\leq\Upsilon_{k_n,n, \ell}\]
%\end{multline*}
where
$$\Upsilon_{k_n,n, \ell}:=r_n^*\sum_{j=1}^{r_n^*-1}\p\left(\A_{n,\ell}\cap \TF^{-j}\A_{n,\ell}\right)
+2t_n\p(\A_{n,\ell})+r_n^*\gamma(q,n,t_n).$$
By \eqref{eq:limit-Anq}, we may assume that $n$ is sufficiently large so that
$\left|1-r_n\p(\A_{n,\ell})\right|<1$ which implies
\begin{align*}
&\left|\p(\BB_{i_{\ell},\ell}\cap\AA_{\ell+1})-\left(1-r_n\p(\A_{n,\ell})\right)\p(\BB_{i_{\ell}+1,\ell}\cap\AA_{\ell+1})\right|\leq\Upsilon_{k_n,n,\ell},
\end{align*}
and
\begin{align*}
\Big|&\p(\BB_{i_{\ell},\ell}\cap\AA_{\ell+1})-\left(1-r_n\p(\A_{n,\ell})\right)^2\p(\BB_{i_{\ell}+2,\ell}\cap\AA_{\ell+1})\Big|\\
&\leq \left|\p(\BB_{i_{\ell},\ell}\cap\AA_{\ell+1})-\left(1-r_n\p(\A_{n,\ell})\right)\p(\BB_{i_{\ell}+1,\ell}\cap\AA_{\ell+1})\right|\\
&\quad+\left|1-r_n\p(\A_{n,\ell})\right|\left|\p(\BB_{i_{\ell}+1,\ell}\cap\AA_{\ell+1})-\left(1-r_n\p(\A_{n,\ell})\right)\p(\BB_{i_{\ell}+2,\ell}\cap\AA_{\ell+1})\right|\\
&\leq2\Upsilon_{k_n,n,\ell}.
\end{align*}
Inductively, we obtain
\begin{align*}
&\Big|\p(\BB_{i_{\ell},\ell}\cap\AA_{\ell+1})-\left(1-r_n\p(\A_{n,\ell})\right)^{S_{\ell}}\p(\AA_{\ell+1})\Big|\leq S_{\ell}\Upsilon_{k_n,n,\ell}.
\end{align*}
Using Lemma \ref{lem:time-gap-1},
\begin{align*}
&\Big|\p(\AA_1)-\left(1-r_n\p(\A_{n,1})\right)^{S_1}\p(\AA_2)\Big|\leq \Big|\p(\AA_1)-\p(\BB_{i_1,1}\cap\AA_2)\Big|+\\
&\quad+\Big|\p(\BB_{i_1,1}\cap\AA_2)-\left(1-r_n\p(\A_{n,1})\right)^{S_1}\p(\AA_2)\Big|\\
&\leq 2r_n \p(\A_{n,1})+S_1\Upsilon_{k_n,n,1}.
\end{align*}
and
\begin{align*}
&\Big|\p(\AA_1)-\left(1-r_n\p(\A_{n,1})\right)^{S_1}\left(1-r_n\p(\A_{n,2})\right)^{S_2}\p(\AA_3)\Big|\leq\\
&\quad\Big|\p(\AA_1)-\left(1-r_n\p(\A_{n,1})\right)^{S_1}\p(\AA_2)\Big|+\left|1-r_n\p(\A_{n,1})\right|\Big|\p(\AA_2)-\left(1-r_n\p(\A_{n,2})\right)^{S_2}\p(\AA_3)\Big|\\
&\leq 2r_n \p(\A_{n,1})+S_1\Upsilon_{k_n,n,1}+2r_n \p(\A_{n,2})+S_2\Upsilon_{k_n,n,2}.
\end{align*}

Again, by induction, we obtain
\begin{align*}
&\Big|\p(\AA_1)-\prod_{\ell=1}^m\left(1-r_n\p(\A_{n,\ell})\right)^{S_{\ell}}\Big|\leq 2r_n \sum_{\ell=1}^m\p(\A_{n,\ell})+{\sum_{\ell=1}^{m}S_{\ell}}\Upsilon_{k_n,n,\ell}.
\end{align*}
Now, it is easy to see that $S_\ell\sim k_n|J_\ell|$, for each
$\ell=1,\ldots,m$. Consequently, by \eqref{eq:limit-Anq}, we have
\begin{align*}
&\lim_{n\rightarrow+\infty}
\prod_{\ell=1}^m\left(1-r_n\p(\A_{n,\ell})\right)^{S_{\ell}}=\lim_{n\rightarrow+\infty}\prod_{\ell=1}^m\left(1-\left\lfloor\frac{n}{k_n}\right\rfloor\p(\A_{n,\ell})\right)^{k_n|J_\ell|}
=\prod_{\ell=1}^m\e^{- \nu(A_\ell)|J_\ell|}
\end{align*}

To conclude the proof it suffices to show that
\[
\lim_{n\rightarrow+\infty}2r_n \sum_{\ell=1}^m\p(\A_{n,\ell})+{\sum_{\ell=1}^{m}S_{\ell}}\Upsilon_{k_n,n,\ell}=0.
\]
Since by \eqref{eq:limit-Anq}, we have $n\p(\A_{n,\ell})\to\nu(A_\ell)\geq 0$, and recalling that $k_n\to\infty$ then
\[
\lim_{n\rightarrow+\infty}2r_n\sum_{\ell=1}^m
\p(\A_{n,\ell})=\lim_{n\rightarrow+\infty}\frac{2}{k_n}\sum_{\ell=1}^m
n\p(\A_{n,\ell})=0.
\]
Next we need to check that
$
\lim_{n\rightarrow+\infty}
{\sum_{\ell=1}^{m}S_{\ell}}\Upsilon_{k_n,n,\ell}=0,
$
which means,
\begin{equation*}
r_n^*\sum_{\ell=1}^{m}k_n|J_\ell|\sum_{j=1}^{r_n^*-1}\p\left(\A_{n,\ell}\cap \TF^{-j}\A_{n,\ell}\right)
+2t_n\sum_{\ell=1}^{m}k_n|J_\ell|\p(\A_{n,\ell})+r_n^*\sum_{\ell=1}^{m}k_n|J_\ell|\gamma(q,n,t_n)\rightarrow0.
\end{equation*}
For the first term, we observe that
$$
r_n^*\sum_{\ell=1}^{m}k_n|J_\ell|\sum_{j=1}^{r_n^*-1}\p\left(\A_{n,\ell}\cap \TF^{-j}\A_{n,\ell}\right)\leq mHn\sum_{j=1}^{r_n^*-1}\p\left(\A_{n,\ell}\cap \TF^{-j}\A_{n,\ell}\right),
$$
which vanishes by condition $\D'_q(u_n)$.
The second term also vanishes because, by \eqref{eq:kn-sequence}, we have $k_nt_n=o(n)$ and, by \eqref{eq:limit-Anq}, we have $n\p(\A_{n,\ell})\to\nu(A_\ell)\geq 0$. The third term vanishes on account of condition $\D_q^*(u_n)$.
\end{proof}

\section{Applications to dynamical systems}

In this section we start by discussing the properties of the systems needed in order to apply the theory developed earlier. Then we list classes of systems that have these properties and, finally, we give some concrete examples where we compute the exactly the limiting processes.

\subsection{Properties of the systems}\label{subsec:properties-system}

The three main ingredients needed to apply Theorem~\ref{thm:indep-increments} are conditions $\D_q^*(u_n)$, $\D'_q(u_n)$ and the outer measure $\nu$ given by \eqref{eq:limit-Anq}. 
\subsubsection{Condition $\D_q^*(u_n)$}\label{subsubsec:D_q^*}
This a mixing type condition and can be shown to follow from certain decay of correlations conditions. 
\begin{definition}[Decay of correlations]
\label{def:dc}
Let \( \mathcal C_{1}, \mathcal C_{2} \) denote Banach spaces of real valued measurable functions defined on \( \X \).
We denote the \emph{correlation} of non-zero functions $\phi\in \mathcal C_{1}$ and  \( \psi\in \mathcal C_{2} \) w.r.t.\ a measure $\mu$ as
\[
\cv_\mu(\phi,\psi,n):=\frac{1}{\|\phi\|_{\mathcal C_{1}}\|\psi\|_{\mathcal C_{2}}}
\left|\int \phi\, (\psi\circ f^n)\, \dif\mu-\int  \phi\, \dif\mu\int
\psi\, \dif\mu\right|.
\]

We say that we have \emph{decay
of correlations}, w.r.t.\ the measure $\mu$, for observables in $\mathcal C_1$ \emph{against}
observables in $\mathcal C_2$ if, for every $\phi\in\mathcal C_1$ and every
$\psi\in\mathcal C_2$ we have
 $\cv_\mu(\phi,\psi,n)\to 0,\,\text{ as $n\to\infty$.}$
  \end{definition}

We say that we have \emph{decay of correlations against $L^1$
observables} whenever  this holds for $\mathcal C_2=L^1(\mu)$  and
$\|\psi\|_{\mathcal C_{2}}=\|\psi\|_1=\int |\psi|\,\dif\mu$.

The main advantage of condition $\D_q^*(u_n)$ when compared to the $\Delta$ condition, which is used in the classical literature on the subject, is the fact that, unlike $\Delta$, it follows easily from summable decay of correlations, \ie $\cv_\mu(\phi,\psi,n)\leq \rho_n$, where $\sum_{n=1}^\infty\rho_n<\infty$. This is easily seen if $\I_A\in \mathcal C_1, \mathcal C_2$ for every $A\in\mathcal B$ because then, by taking $\phi=\I_{\A_{n,k}}$ and $\psi=\I_{ \bigcap_{i=k}^m\W_{J_{n,i}-t}\left(\A_{n,i}\right)}$, as long as $\|\phi\|_{\mathcal C_1}, \|\psi\|_{\mathcal C_2}\leq C>0$ for all $n\in\N$, then we immediately get that summable decay of correlations for observables in $\mathcal C_1$ against
observables in $\mathcal C_2$ implies condition $\D_q^*(u_n)$, where $\gamma(q,n,t)=C^2\rho_t$ because then we can choose $(t_n)_{n\in\N}$ such that $t_n=o(n)$ and $\lim_{n\to\infty}C^2\rho_{t_n}=0$.
\begin{remark}
We remark that $\mathcal C_1$ does not necessarily need to contain the functions $\I_{\A_{n,k}}$. For example, for one-dimensional systems if $\mathcal C_1$ is the space of H\"older continuous functions we can use a continuous approximation of $\I_{\A_{n,k}}$ and still prove condition $\D_q^*(u_n)$. See discussion in \cite[Section~5.1]{F13} or \cite[Section~4.4]{LFFF16}. For higher dimensional systems with contracting directions, for example, one can still check $\D_q^*(u_n)$ following the argument used in \cite[Section~2]{GHN11}. 
\end{remark}
\subsubsection{Condition $\D'_q(u_n)$}\label{subsubsec:D'_q} This condition is the hardest to check and depends on the short recurrence properties of small vicinities  of maximal set, \ie the set of points where $\varphi$ is maximised. Here, we are assuming that the maximal set is reduced to a point but the theory can be extended using the ideas in \cite{AFFR16,AFFR17} for more general maximal sets. For typical points $\zeta$, \ie for $\mu$-a.e. $\zeta$, we have that condition $\D'_0(u_n)$ holds, which means absence of clustering. This can be proved for different systems as in \cite{C01,CC13, GHN11, HNT12, FHN14}. At periodic points $\zeta$, condition $\D'_q(u_n)$, with $q$ equal to the period of $\zeta$, can be proved to hold for systems with a strong form of decay of correlations, namely, summable decay of correlations against $L^1$, as in \cite[Lemma~4.1]{FFT12} or \cite[Section~4.2.4]{LFFF16}. In fact, under this assumption, the periodic points are the only exceptions and one can show that for all non-periodic points condition $\D'_0(u_n)$ holds, \cite{AFV15}, which gives a dichotomy.

\subsubsection{The outer measure}\label{subsubsec:outer-measure}
Even after assuring the existence of $\nu$, its computation is quite sensitive. In fact, as observed in \cite{CFFH15}, with the formulas for the EI and for the multiplicity distribution of 1-dimensional REPP, the choice of the metric may affect the computation. Hence, in order to find the outer measure that describes the stacking of mass points of the limiting point process, one needs some regularity of the system. To be more precise, one needs some regularity of the measure $\mu$ and of the action of $T$ on a small neighbourhood of the maximal set, $\zeta$. In order to give precise formulas for $\nu$ in some worked out examples, we will assume that $\mu$ is absolutely continuous with respect to Lebesgue measure and  its Radon-Nikodym density is sufficiently regular so that for all $x\in\X$ we have
\begin{equation}
\label{eq:density}
\lim_{\eps\to 0}\frac{\mu(B_\eps(x))}{|B_\eps(x)|}=\frac{d\mu}{d\leb}(x).
\end{equation}
\begin{remark}
Note that if $T$ is one dimensional smooth map modelled by the full shift as in \cite[Section~7.1]{FFT15}
and the derivative is sufficiently regular then, as seen in \cite[Section~7.3]{FFT15}, the invariant density is fairly smooth and formula \eqref{eq:density} holds for all $x\in\X$. Moreover, we can handle equilibrium states that are singular measures w.r.t. Lebesgue if there exists some regularity w.r.t. the conformal measure as described in  \cite[Section~7.3]{FFT15}.
\end{remark}

\subsubsection{Extensions by inducing}\label{subsubsec:inducing}
In \cite{BSTV03}, the authors used inducing to extend the limiting distributional results regarding rare events to systems which admitted a `nice' first return time induced map. This result was proved for generic points $\zeta$ and later generalised, in \cite{HWZ14}, for all points. More recently, in \cite{FFM17}, it was proved that if a one-dimensional marked empirical REPP converges to certain random measure (marked point process) for the action of a first return time induced map of some system then the same statement holds true for the original dynamics of the system. The proof \cite[Theorem~2.C, Section~5]{FFM17} can be adapted in order to give the same statement for the multi-dimensional point processes we consider here. Namely, replacing the event $\mathbb A(J,\mathbf x,n)=\{A_n(I_1)>x_1, A_n(I_2)>x_2,\ldots, A_n(I_k)>x_k\}$, defined in the beginning of the proof of \cite[Lemma~5.1]{FFM17}, where $J$ is the disjoint union of the time intervals $I_j$ and $\mathbf x=(x_1,\ldots,x_k)\in\R^k$ and the marked point process $A_n$ is defined in \cite[Definition~2.7]{FFM17}, by 
$$
\mathbb A(E,\mathbf l,n)=\{N_n(J_1\times A_1)=l_1, N_n(J_1\times A_1)=l_2,\ldots, N_n(J_m\times A_m)=l_m\},
$$
where we are using the same notation employed in the beginning of Section~\ref{subsec:dependence} for $E$ and $\mathbf l=(l_1,\ldots,l_m)\in\N_0^m$, replacing $v_n$ by $n$ and  $r_{U_n}$ by $r_{A_{n,k}}$, for each $k=1, \ldots,m$, the proof goes through practically with the same steps.

\subsection{Classes of systems of application}\label{subsec:systems}
The first class of systems we consider are those for which there is summable decay of correlations against $L^1$. Examples of systems with such property include:
\begin{enumerate}[label=A.{\arabic*})]

\item Uniformly expanding maps on the circle/interval (see \cite{BG97});

\item Markov maps (see \cite{BG97});

\item \label{rychlik}Piecewise expanding maps of the interval with countably many branches like Rychlik maps (see \cite{R83});

\item Higher dimensional piecewise expanding maps studied by Saussol in \cite{S00}.

\end{enumerate}

For all the systems above, a dichotomy holds, \ie conditions $\D^*_q(u_n)$ and $\D'_q(u_n)$ hold with $q=0$, for non-periodic points, $\zeta$, while, for periodic points, $q$ can be taken equal to a multiple of the period of $\zeta$. Then, for all systems which admit a first return time induced map, which falls into the category above, we can also apply the convergence results. Examples of such systems with nice induced maps include:
\begin{enumerate}[label=B.{\arabic*})]

\item quadratic maps with invariant densities for which the critical point is non-recurrent, such as Misiurewicz maps (see \cite{M81,MS93})

\item maps with indifferent fixed points such as the Manneville-Pommeau or Liverani-Saussol-Vaienti maps (see \cite{PM80,LSV99})

\end{enumerate}

In the case of absence of clustering, \ie when $q=0$, for generic points $\zeta$, more examples can be considered, namely,
\begin{enumerate}[label=D.{\arabic*})]

\item toral automorphisms  (see \cite{CFFH15})

\item H\'enon maps (see \cite{CC13})

\item Billiards with exponential and polynomial decay of correlations (\cite{GHN11,FHN14})

\item General non-uniformly hyperbolic systems admitting Young towers with polynomial tails (\cite{PS16,HW16}).

\end{enumerate}

\subsection{A dichotomy}\label{subsec:dichotomy}
In order to illustrate the scope of application we state and prove a general theorem giving the convergence of empirical multi-dimensional REPP. We remark that some of the conditions are not necessary and in some situations they can be relaxed. Yet, this result can be applied to all the examples A.1)--A.4), directly. Moreover, since the examples B.1)--B.2) have a first-return time induced map that falls into category \ref{rychlik}, then the convergence of the multi-dimensional REPP can be obtained indirectly from this result as described in Section~\ref{subsubsec:inducing}. 
\begin{theorem}
\label{thm:convergence-hyperbolic}
Let $T:\X\to\X$ be a system with an invariant measure $\mu$ satisfying \eqref{eq:density} and summable decay of correlations against $L^1$ observables, \ie for all $\phi\in\mathcal C_1$ and $\psi\in L^1$, then $\cv(\phi,\psi,n)\leq \rho_n$, with $\sum_{n\geq 1}\rho_n<\infty$, where $\mathcal C_1$ contains functions of the type $\phi=\I_B$, for all $B\in\mathcal B$ and if $(B_n)_{n\in\N}$ is such that there exists a uniform bound for the number of connected components of all $B_n\in\mathcal B$, then there exists $C>0$ such that $\|\I_{B_n}\|_{\mathcal C_1}\leq C$, for all $n\in\N$. Consider the stochastic processes given by \eqref{eq:SP-def} for observables of the type \eqref{eq:observable-form} and normalising sequences $(u_n(\tau))_{n\in\N}$ as in \eqref{un}. Assume that $\zeta$ is an hyperbolic repelling point such that $T$ is continuous at all points of the orbit of $\zeta$. Then we have the following statements:
\begin{enumerate}

\item if $\zeta$ is a non-periodic point then the 2-dimensional empirical REPP process $N_n$ given in Definition~\ref{def:REPP-2d} converges to the 2-dimensional Poisson point process $N$ given by \eqref{Poisson-2d} and the multi-dimensional empirical REPP process $N^*_n$ given in Definition~\ref{def:multi-dimensional-REPP} converges to a multi-dimensional Poisson process $N^*$ given by \eqref{eq:repelling-limit-REPP-multi-d-no-clustering}.

\item if $\zeta$ is a periodic point then the multi-dimensional empirical REPP process $N^*_n$ given in Definition~\ref{def:multi-dimensional-REPP} converges to a multi-dimensional Poisson process $N^*$ as in \eqref{eq:repelling-limit-REPP-multi-d} and if $\zeta$ expands uniformly  in all directions then the 2-dimensional empirical REPP process $N_n$ given in Definition~\ref{def:REPP-2d} converges to the process $\tilde N$ given by \eqref{eq:repelling-limit-REPP-2d}.

\end{enumerate}
 
\end{theorem}

As direct consequence we recover the dichotomy regarding the convergence of the 1-dimensional REPP studied in \cite{FFT13,AFV15}. In fact, note that the corollary below implies, for example, the main statement of \cite[Theorem~1]{FFT13} and \cite[Proposition~3.3]{AFV15}.

\begin{corollary}
\label{cor:dichotomy}
Under the same conditions of Theorem~\ref{thm:convergence-hyperbolic}, we have:
\begin{enumerate}

\item if $\zeta$ is a non-periodic point then the 1-dimensional empirical REPP process $N_n^1$ given in Definition~\ref{def:1d-REPP} converges to the 1-dimensional homogeneous Poisson process $N$ of intensity 1.

\item if $\zeta$ is a periodic point then the 1-dimensional empirical REPP process $N_n^1$ given in Definition~\ref{def:1d-REPP} converges to the compound Poisson process $\tilde N^1$ given in \eqref{eq:CPP}, with multiplicity distribution given by $\pi(\kappa)=\p(D_i=\kappa)
=\theta(1-\theta)^{\kappa-1}$.

\end{enumerate}

\end{corollary}

We remark that the convergence of the multi-dimensional REPP $N_n^*$ to $N^*$ allows us to obtain the convergence of 2-dimensional REPP $N_n$ even when the expansion is not uniform in all directions, although the limiting process is now more complicated to describe.

\begin{corollary}
\label{cor:two-dimensional}
Under the same conditions of Theorem~\ref{thm:convergence-hyperbolic}, in the case of $\zeta$ being a periodic hyperbolic repelling point, regarding the convergence of the 2-dimensional empirical REPP process $N_n$ given in Definition~\ref{def:REPP-2d} we have: $N_n$ converges to the point process $N^\dag$ given by:
\begin{equation}
\label{eq:N-dag}
N^\dag=\sum_{i,j=1}^\infty\sum_{\ell=0}^{\infty} \delta_{\left(T_{i,j},\,\left|B_{\dist\left(\Phi_\zeta\left((U_{i,j}/|B_1(\zeta)|)^{1/d}DT_\zeta^{\ell p}( \cos\Theta_{i,j},\sin\Theta_{i,j})\right),\zeta\right)}(\zeta)\right|\right)},
\end{equation}
where  $T_{i,j}$ is given by \eqref{T-matrix}, for $\bar T_{i,j}\stackrel[]{D}{\sim}\text{Exp}\left(\theta\right)$, with $\theta=1-|\det^{-1}(DT^p_\zeta)|$; $U_{i,j}\stackrel[]{D}{\sim} \mathcal U_{(i-1,i]}$; $\Theta_{i,j}\sim \U_{[0,2\pi)}$ and $(\bar T_{i,j})_{i,j\in\N}$, $(U_{i,j})_{i,j\in\N}$ and $(\Theta_{i,j})_{i,j\in\N}$ are mutually independent. 
\end{corollary}

\begin{proof}[Proof of Theorem~\ref{thm:convergence-hyperbolic}]
 Let $J_k$, $A_k$, $J_{n,k}$, $A_{n,k}$ and $\A_{n,k}$ be fixed as in the beginning of Section~\ref{subsec:dependence}. We start by noting that, for both situations, condition $\D_q^*(u_n)$ is proved directly from decay of correlations with the choices for $\phi$ and $\psi$ indicated in Section~\ref{subsubsec:D_q^*}. 
 
For the proof of $\D'_q(u_n)$  we use the fact that decay of correlations holds for all $\psi\in L^1$. We take $\phi=\psi=\I_{\A_{q,n}}$ and since $\|\I_{\A_{q,n}}\|_{\mathcal C_1}\leq C$, for some $C>0$, we have
\begin{align}
\label{eq:estimate1}
\mu\left(\A_{n,k}\cap T^{-j}(\A_{n,k})\right) \le
 %\left(\mu(\A_{n,k})\right)^2+ \left\| \I_{\A_{n,k}}\right\|_{\mathcal C_1} \left\| \I_{\A_{n,k}}\right\|_{L^1(\mu)} \rho_j\leq  
 \left(\mu(\A_{n,k})\right)^2+C\mu(\A_{n,k})\rho_j.
\end{align}
Recall that by \eqref{def:q}, $q$ is chosen so that $\lim_{n\to\infty}R_n^{(q)}=\infty$. Using estimate \eqref{eq:estimate1} it follows that there exists some constant $D>0$ such that
\begin{align*}
n\sum_{j=q+1}^{\lfloor n/k_n \rfloor}& \mu(\A_{n,k}\cap T^{-j}(\A_{n,k})) = n\sum_{j=R_n^{(q)}}^{\lfloor n/k_n \rfloor} \mu(\A_{n,k}\cap T^{-j}(\A_{n,k}))\\
&%\le n\big\lfloor\tfrac {n}{k_n}\big\rfloor\mu(\A_{n,k})^2 +n\,C\mu(\A_{n,k}) \sum_{j=R_n^{(q)}}^{\lfloor n/k_n \rfloor}\rho_j
\le \frac{\left(n\mu(\A_{n,k})\right)^2}{k_n} +n\,C\mu(\A_{n,k}) \sum_{j=R_n^{(q)}}^{\infty}\rho_j \leq D \Bigg(\frac{|A_k|^2}{k_n}+|A_k| \sum_{j=R_n^{(q)}}^{\infty}\rho_j \Bigg)\xrightarrow[n\to\infty]{}0.
\end{align*}
 
We consider now the first case and show that $q=0$ and $\nu=\leb$. To prove that $q=0$, we need to show that $\lim_{n\to\infty}R_n^{(0)}=\infty$, which follows by a simple continuity argument, since $T$ is continuous along the orbit of $\zeta$, which is a non-periodic point (see \cite[Lemma~3.1]{AFV15}). Since $q=0$, we have $\A_{n,k}=A_{n,k}$. Recalling that $\lim_{n\to\infty}n\mu(A_{n,k})=|A_k|$, then formula \eqref{eq:limit-Anq} gives that $\nu=\leb$. %Note that $c=1$, when $A_{n.k}$ is given by \eqref{eq:rectangles-n}. 
Observe that property \eqref{eq:density} of the measure $\mu$ was not needed in this case.

Let us consider first the case of the two-dimensional empirical REPP $N_n$, where  that $A_{n,k}$ is given by \eqref{eq:rectangles-n}. In this case, $n\mu(A_{n,k})\sim|A_k|$. Let us assume now that $\zeta$ is a repelling periodic point of period $p$, which expands uniformly in all directions, namely, there exists some $\alpha>1$ such that $DT^p_\zeta(v)=\alpha \,v$, for $v\in T_\zeta \X$. This implies that if $T^j(x)$ is very close to $\zeta$ then $\dist(T^{j+p}(x),\zeta)\approx \alpha\, \dist(T^j(x),\zeta)$. 
Observe that under \eqref{eq:density}, for $n$ large, $\mu$ behaves as Lebesgue measure and, hence, 
\begin{equation}
\label{eq:fact1}
n\mu(T^{-p}A_{n,k})\sim \alpha^{-d} |A_k|=|\alpha^{-1} A_k|.
\end{equation} 
Recall that  $A_{n,k}=\cup_{s=1}^{\varsigma_k} G_{n,k,s}$, where each $G_{n,k,s}$ corresponds to an annulus around $\zeta$ given by $G_{n,k,s}=B_{g^{-1}(u_n(\tau_{k,2s-1}))}(\zeta)\setminus B_{g^{-1}(u_n(\tau_{k,2s}))}(\zeta)$ and hence $A_{n,k}$ corresponds to a finite union of annuli around $\zeta$ that as $n$ grows get closer and closer to $\zeta$. In a small neighbourhood of $\zeta$, say $B_{\eps}(\zeta)$, the action of $T^p$ is diffeomorphically conjugate to that of $DT^p_{\zeta}$, which essentially corresponds to repelling the points by a factor $\alpha$. This means that in order to be able to choose an adequate $q$, we need to show that there exists some $q'\in\N$ for which $\lim_{n\to\infty}R_n(A_{n,k}^{(q')})=\infty$. For that purpose, we need to consider that besides the chance of an annulus hitting itself after $p$ iterates,  we have to pay attention to the possibility of the inner annuli, corresponding to smaller $s$, hitting an outer annulus corresponding to large $s$. Therefore, a good candidate for $q'$ would be $q'=p\left \lceil \frac{\log(\tau_{k,2\varsigma_k})-\log(\tau_{k,1})}{d\log\alpha}\right\rceil$, when $\tau_{k,1}>0$. In the case $\tau_{k,1}=0$, it should be replaced by $\alpha^{-d}\tau_{k,2}$, since the ball $B_{g^{-1}(u_n(\alpha^{-1}\tau_{k,2}))}$ has to be removed  from $\A_{n,k}$ to form $A_{n,k}^{(q')}$.   Note that $q/p$ is the smallest integer $j$ such that $\alpha^{-dj}\tau_{k,2\varsigma_k}<\tau_{k,1}$, which means that $\dist(T^{q'}(G_{n,k,1}),\zeta)=\dist(T^{q'}(A_{n,k}),\zeta)>g^{-1}(u_n(\tau_{k,2\varsigma_k}))$. Since by definition of $A_{n,k}^{(q')}$, we have $A_{n,k}^{(q')}\subset A_{n,k}\subset B_{g^{-1}(u_n(\tau_{k,2\varsigma_k}))}(\zeta)$ and $T^j(A_{n,k}^{(q')})\cap A_{n,k}^{(q')}=\emptyset$ for all $j=1, \ldots, q'$. Then, the choice of $q'$, allows to conclude that while $T^j(A_{n,k}^{(q')})\subset B_\eps(\zeta)$, we have that $T^j(A_{n,k}^{(q')})\cap A_{n,k}^{(q')}=\emptyset$. Since $g^{-1}(u_n(\tau_{k,2\varsigma_k}))=O(n^{-d})$, in order for $T^j(A_{n,k}^{(q')})$ to leave $B_\eps(\zeta)$, we need at least $O(\log n)$ iterations, which means $\lim_{n\to\infty}R_n(A_{n,k}^{(q')})=\infty$, where we are using the Bachmann-Landau big $O$ notation.

We now turn to the computation of $\nu$, for which we observe that 
\begin{equation}
\label{eq:Ankq}
A_{n,k}^{(q')}= A_{n,k}\setminus \left(\bigcup_{i=1}^{q'/p} T^{-i\,p} A_{n,k}\right)=A_{n,k}\setminus \left(\bigcup_{i=1}^{+\infty}T^{-i\,p} A_{n,k}\right).
\end{equation} Since for large $n$ the action of $T^p$ is given by that of $DT^p(\zeta)$, as in \eqref{eq:fact1}, we have
$$
\nu(A_k)=\lim_{n\to\infty}n\mu(\A_{n,k})=\left|A_k\setminus\bigcup_{i=1}^{+\infty} \alpha^{-i\cdot d} A_{k}\right|.
$$
Using Theorem~\ref{thm:indep-increments}, it follows that:
\begin{equation}
\label{eq:Ntilde}
\lim_{n\to\infty}\mu\left( N_n\left(\bigcup_{i=1}^m J_k\times A_k\right)=0\right)=\prod_{k=1}^m\e^{-|J_k|\nu(A_k)}.
\end{equation}
Hence, in order to check that $\tilde N$ given by \eqref{eq:repelling-limit-REPP-2d} verifies Criterion (A) of Kellenberg Theorem~\ref{thm:Kallenberg}, we are left to check $\p\left(\tilde N\left(\bigcup_{i=1}^m J_k\times A_k\right)=0\right)=\prod_{i=1}^m\e^{-|J_k|\nu(A_k)}$. Since the $\bar T_{i,j}$ are all independent then we only need to check that for each $k=1, \ldots,m$, we have $\p\left(\tilde N\left(J_k\times A_k\right)=0\right)=\e^{-|J_k|\nu(A_k)}$. By definition of $\tilde N$, it follows
\begin{align*}
\p\Big(\tilde N&\left(J_k\times A_k\right)=0\Big)=\e^{-\theta\tau_{k,2\varsigma_k}|J_k|}+\e^{-\theta\tau_{k,2\varsigma_k}|J_k|}\frac{(\theta\tau_{k,2\varsigma_k}|J_k|)}{1!}\frac{(\tau_{k,2\varsigma_k}-|\cup_{i=0}^\infty \alpha^{-i\cdot d}A_k|)}{\tau_{k,2\varsigma_k}}\\
&\quad+\e^{-\theta\tau_{k,2\varsigma_k}|J_k|}\frac{(\theta\tau_{k,2\varsigma_k}|J_k|)^2}{2!}\frac{(\tau_{k,2\varsigma_k}-|\cup_{i=0}^\infty \alpha^{-i\cdot d}A_k|)^2}{\tau_{k,2\varsigma_k}^2}+\ldots\\
&=\e^{-\theta\tau_{k,2\varsigma_k}|J_k|}\cdot \exp\left\{\frac{\theta\tau_{k,2\varsigma_k}|J_k|}{\tau_{k,2\varsigma_k}}(\tau_{k,2\varsigma_k}-|\cup_{i=0}^\infty \alpha^{-i\cdot d}A_k|)\right\}=\e^{-\theta|J_k||\cup_{i=0}^\infty \alpha^{-i\cdot d}A_k|}
\end{align*}
Now, we only need to verify that $\theta |\cup_{i=0}^\infty \alpha^{-i\cdot d}A_k|=|A_k\setminus\cup_{i=1}^{+\infty} \alpha^{-i\cdot d} A_{k}|$. Since $\theta=1-\alpha^{-d}$, then
\begin{align*}
\theta |\cup_{i=0}^\infty \alpha^{-i\cdot d}A_k|&=|\cup_{i=0}^\infty \alpha^{-i\cdot d}A_k|-\alpha^{-d}|\cup_{i=0}^\infty \alpha^{-i\cdot d}A_k|=|\cup_{i=0}^\infty \alpha^{-i\cdot d}A_k|-|\cup_{i=1}^\infty \alpha^{-i\cdot d}A_k|\\
&=|\cup_{i=0}^\infty \alpha^{-i\cdot d}A_k\setminus\cup_{i=1}^\infty \alpha^{-i\cdot d}A_k|=|A_k\setminus\cup_{i=1}^\infty \alpha^{-i\cdot d}A_k|.
\end{align*}

At this point, we check that $\tilde N$ also satisfies Kallenberg's criterion (B). Let $E=J\times G=[a,b)\times [\tau_1,\tau_2)$. Then,
\begin{align}
\E(N_n(E))&=\E\left(\sum_{j\geq0}\delta_{(j/n,u_n^{-1}(X_j))}(E)\right)=\E\left(\sum_{j=\lceil na\rceil}^{\lfloor nb\rfloor}\I_{u_n(\tau_2)<X_j\leq u_n(\tau_1)}\right)\nonumber \\
&\sim% n(b-a)\mu(u_n(\tau_2)<X_j\leq u_n(\tau_1))=
(b-a)n(\mu(X_j> u_n(\tau_2))-\mu(X_j> u_n(\tau_1)))\xrightarrow[n\to\infty]{}(b-a)(\tau_2-\tau_1).
\label{eq:Kallenberg2}
\end{align}
Thus, we need to check that $\E(\tilde N(E))=|E|=|J| |G|=(b-a)(\tau_2-\tau_1)$.

Let $N$ stand for the 2-dimensional homogeneous Poisson process given by \eqref{Poisson-2d} (such that $\E(N(E))=\leb(E)$). Then rewriting $\tilde N$ as
\[\tilde N=\sum_{i,j=1}^\infty\sum_{\ell=0}^{\infty}\delta_{(\theta^{-1}T_{i,j},\,\alpha^{\ell d}\,U_{i,j})}\]
with $\bar T_{i,j}\stackrel[]{D}{\sim}\text{Exp}(1)$, for $E=J\times G$, it follows that:
\begin{align*}
\E(\tilde N(E))&=\E\left(\sum_{i,j=1}^\infty\sum_{\ell=0}^{\infty}\delta_{(\theta^{-1}T_{i,j},\,\alpha^{\ell d}\,U_{i,j})}(E)\right)=\sum_{\ell=0}^{\infty}\E\left(\sum_{i,j=1}^\infty\delta_{(\theta^{-1}T_{i,j},\,\alpha^{\ell d}\,U_{i,j})}(J\times G)\right)\\
&=\sum_{\ell=0}^{\infty}\E\left(\sum_{i,j=1}^\infty\delta_{(T_{i,j},\,U_{i,j})}(\theta J\times\alpha^{-\ell d}G))\right)=\sum_{\ell=0}^{\infty}\E(N(\theta J\times\alpha^{-\ell d}G))\\
&=\sum_{\ell=0}^{\infty}\leb(\theta J\times\alpha^{-\ell d}G)=\sum_{\ell=0}^{\infty}\theta\alpha^{-\ell d}\leb(J\times G)=\frac{\theta\leb(E)}{1-\alpha^{-d}}=\leb(E).
\end{align*}
Let us consider now the case of the empirical multi-dimensional REPP, $N_n^*$, given in Definition~\ref{def:multi-dimensional-REPP}. In this case, $\zeta$ is an hyperbolic repelling point but does not necessarily expand  uniformly in all directions. As before, since for large $n$ the normalisation used in \eqref{eq:rectangles-n-bi} implies that $A_{n,k}$ is contained in ball of radius $O(n^{-1/d})$ around $\zeta$, then the action of $T^p$ is approximated by that of $DT^p_\zeta$, but now the latter is not necessarily described by a multiplication by the factor $\alpha$, since now the expansion rates may be different in different directions. Of course that if $\X$ is one-dimensional, then we are again reduced to the previous case. The argument flows in the same way with slight adjustments. Using \eqref{eq:density}, we observe that the normalisation was made such that
\begin{align}
n\mu(A_{n,k})&\sim n\frac{d\mu}{d\leb}(\zeta)\left|g^{-1}(u_n(1))|B_1(\zeta)|^{1/d}A_{k}\right|\sim n\frac{d\mu}{d\leb}(\zeta)|B_1(\zeta)| (g^{-1}(u_n(1)))^d |A_{k}|\nonumber\\
&\sim n\frac{d\mu}{d\leb}(\zeta)\left|B_{g^{-1}(u_n(1))}(\zeta)\right||A_{k}|\sim n\mu\left(B_{g^{-1}(u_n(1))}(\zeta)\right) |A_{k}|\sim  |A_{k}|.
\label{eq:normalisation}
\end{align}

Let $\gamma_n^+=\sup\{\dist(x,\zeta):\,x\in \cup_{k=1}^m A_{n.k}\}$, $\gamma_n^-= \inf\{\dist(x,\zeta):\,x\in \cup_{k=1}^m A_{n.k}\}$. Also let $\beta^+$ and $\beta^-$ denote, respectively, the largest and smallest, in absolute value, eigenvalues  of $DT^p(\zeta)$. In this case, we take $q'=p\lceil\frac{\log(\gamma_n^+)-\log(\gamma_n^-)}{\log\beta^-}\rceil$, when $\gamma_n^->0$. In the case $\gamma_n^-=0$, we replace it by $(\beta^+)^{-1}\gamma_n^\partial$, where $\gamma_n^\partial= \inf\{\dist(x,\zeta):\,x\in \cup_{k=1}^m \partial A_{n.k}\}$ and $\partial A_{n.k}$ denotes the border of $A_{n.k}$. Let $x$ be such that $\dist(x,\zeta)=\gamma_n^-$. This choice of $q'$ guarantees that $\dist(T^{q'p}(x),\zeta)>\gamma_n^+$. Hence, as in the previous case, it follows that $\lim_{n\to\infty}R_n(A_{n,k}^{(q')})=\infty$. Observe that in this case, formula \eqref{eq:Ankq} still holds and since for large $n$ the action of $T^p$ is given by that of $DT^p(\zeta)$ and, by linearity of $DT^p(\zeta)$, we have $A_{n,k}\setminus \left(\bigcup_{i=1}^{+\infty}DT_\zeta^{-i\,p} A_{n,k}\right)=g^{-1}(u_n(1))|B_1(\zeta)|^{1/d}\left(A_{k}\setminus \left(\bigcup_{i=1}^{+\infty}DT_\zeta^{-i\,p} A_{k}\right)\right)$, then, in the same away as in the estimate obtained for $n\mu(A_{n,k})$, we get: 
$$
\nu(A_k)=\lim_{n\to\infty}n\mu(\A_{n,k})=\left|A_k\setminus\bigcup_{i=1}^{+\infty} DT_\zeta^{-i \,p} (A_{k})\right|.
$$
\begin{figure}[h]
\includegraphics[height=5cm]{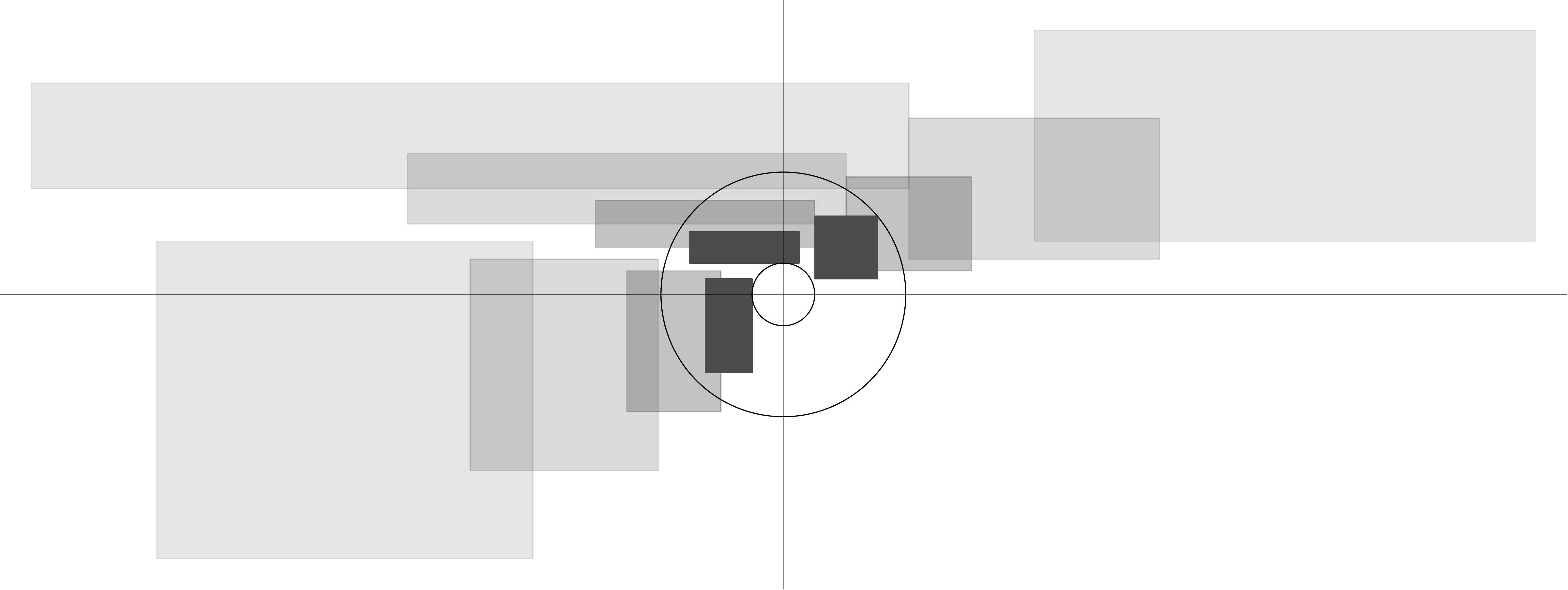}
\caption{In this picture, for some nice coordinates system chosen, the set $A_{n.k}$ corresponds to the darker rectangles, whose subsequent images by $DT_\zeta^p$ are depicted in lighter and lighter tones. The circles have radii $\gamma_n^+>\gamma_n^-$, respectively. The superpositions correspond to portions that need to be removed from $A_{n.k}$ to form $A_{n.k}^{(q')}$. Note that the lighter rectangles do not intersect $A_{n.k}$, which means that we can take $q'=3p$, in this case.}
\label{fig:ring-notation}
\end{figure}

In order to check that $N^*$ given by \eqref{eq:repelling-limit-REPP-multi-d} is consistent with this formula for $\nu$, we observe that
\begin{align*}
\p\Big(N^*\left(J_k\times A_k\right)=0\Big)&=\e^{-\theta|B_{\gamma^+}(\zeta)||J_k|}\cdot \exp\left\{\frac{\theta|B_{\gamma^+}(\zeta)||J_k|}{|B_{\gamma^+}(\zeta)|}|B_{\gamma^+}(\zeta)\setminus\cup_{i=0}^\infty  DT_\zeta^{-i \,p} (A_{k})|\right\}\\
&=\e^{-\theta|J_k||\cup_{i=0}^\infty  DT_\zeta^{-i \,p} (A_{k})|}
\end{align*}
where $\gamma^+=\sup\{\dist(x,\zeta):\,x\in \cup_{i=1}^m A_k\}$, which means that $\cup_{i=0}^\infty  DT_\zeta^{-i \,p} (A_{k})\subset B_{\gamma^+}(\zeta)$. Recalling that $\theta=1-|\det(DT^{-p}_\zeta)|$, we have
\begin{align*}
\theta |\cup_{i=0}^\infty  DT_\zeta^{-i \,p} &(A_{k})|=|\cup_{i=0}^\infty  DT_\zeta^{-i \,p} (A_{k})|-|\det(DT^{-p}_\zeta)||\cup_{i=0}^\infty  DT_\zeta^{-i \,p} (A_{k})|\\
&=|\cup_{i=0}^\infty  DT_\zeta^{-i \,p} (A_{k})|-|\cup_{i=1}^\infty  DT_\zeta^{-i \,p} (A_{k})|=|A_k\setminus\cup_{i=1}^\infty  DT_\zeta^{-i \,p} (A_{k})|.
%&=|A_k\setminus\cup_{i=1}^\infty \alpha^{-i\cdot d}A_k|+|\cup_{i=1}^\infty \alpha^{-i\cdot d}A_k|-|\cup_{i=1}^\infty \alpha^{-i\cdot d}A_k|=|A_k\setminus\cup_{i=1}^\infty \alpha^{-i\cdot d}A_k|
\end{align*}
This proves Kallenberg's criterion (A) for $N^*$. Now, we turn to criterion (B). Let $E=J\times  G=[a.b)\times([e_1,f_1)\times\ldots\times[e_d,f_d)).$ By stationarity and repeating the argument in \eqref{eq:normalisation}, we obtain
\begin{align*}
\E(N_n^*(E))&=\E\left(\sum_{j=\lceil na\rceil}^{\lfloor nb\rfloor}\I_{\Phi_\zeta\left(g^{-1}(u_n(1))|B_1(\zeta)|^{1/d}G\right)}\circ T^j\right)\\
&\sim(b-a)n\mu\left(\Phi_\zeta\left(g^{-1}(u_n(1))|B_1(\zeta)|^{1/d}G\right)\right)\xrightarrow[n\to\infty]{}(b-a)|G|.
\end{align*}
In order to check that $\E(N^*_\theta(E))=(b-a)|G|$, we start by defining the following sequences sets. Let $G^{1*}=G\cap DT_\zeta^{-p}(G)$ and, for each $j\geq 2$, set $G^{j*}=G^{j-1*}\cap DT_\zeta^{-p}(G^{j-1*})$. Then, let $G^{(0)}=B_{\gamma^+}(\zeta)\setminus \cup_{i\geq 0} DT_\zeta^{-ip}(G)$, $G^{(1)}=\cup_{i\geq 0} DT_\zeta^{-ip}(G)\setminus \cup_{i\geq 0} DT_\zeta^{-ip}(G^{1*})$ and, for each $j\geq 2$, set $G^{(j)}=\cup_{i\geq 0} DT_\zeta^{-ip}(G^{j-1*})\setminus \cup_{i\geq 0} DT_\zeta^{-ip}(G^{j*})$. 

Observe that if $0\notin \bar G$, then there exist an $m\in\N$ such that $G^{j*}=\emptyset$ for all $j\geq m$ and, consequently, $G^{(j)}=\emptyset$ for all $j\geq m+1$. Note also that $(G^{(j)})_{j\in\N_0}$ forms a disjoint partition of $\cup_{i\geq 0} DT_\zeta^{-ip}(G)$. Moreover, we have (see Figure~\ref{fig:decomposition})
\begin{equation}
\label{eq:return-average-formula}
\sum_{j=0}^\infty \left| DT_\zeta^{-jp}(G)\right|=\sum_{j=1}^\infty j \left|G^{(j)}\right|.
\end{equation}
\begin{figure}[h]
\includegraphics[width=6cm]{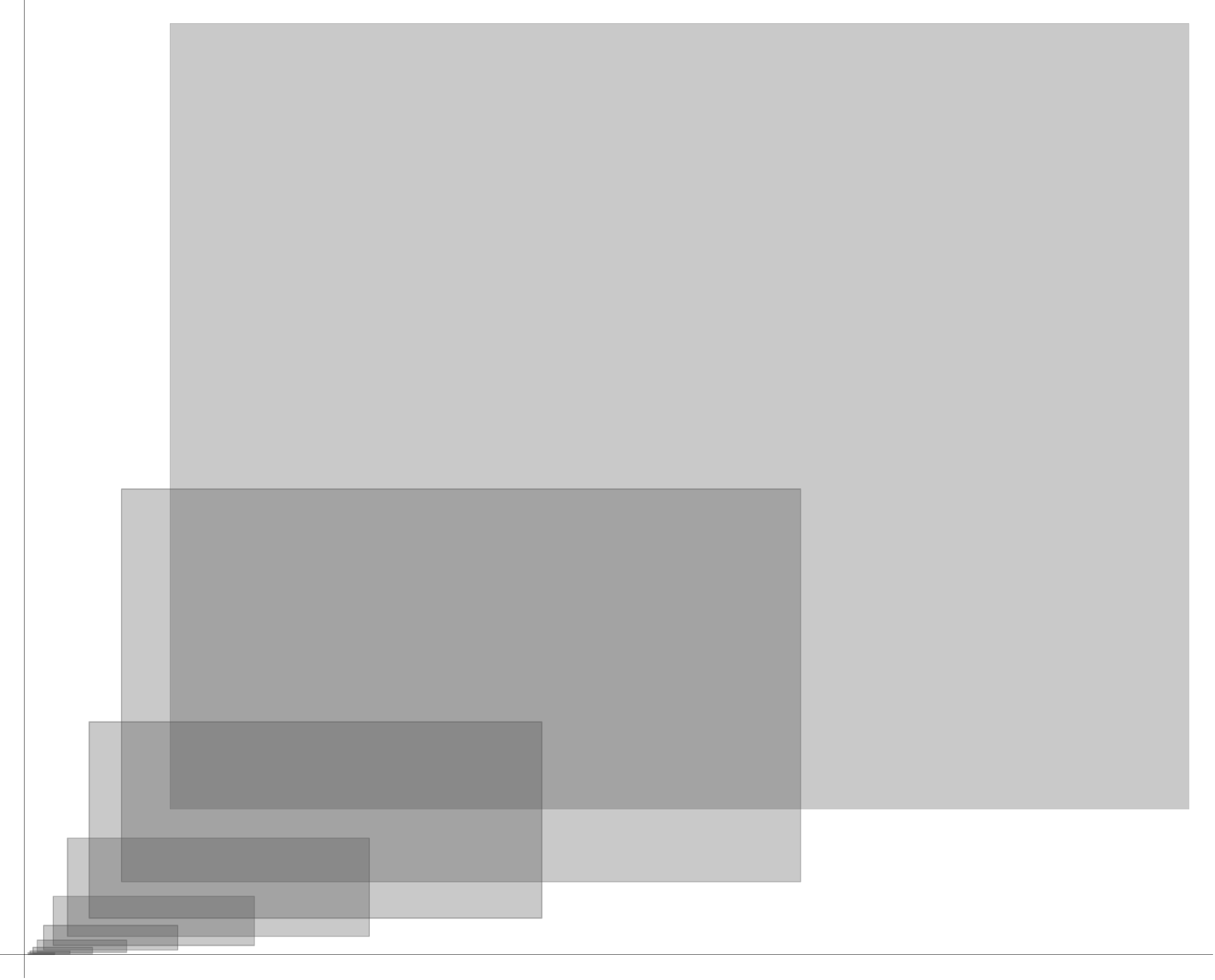}\quad\includegraphics[width=6cm]{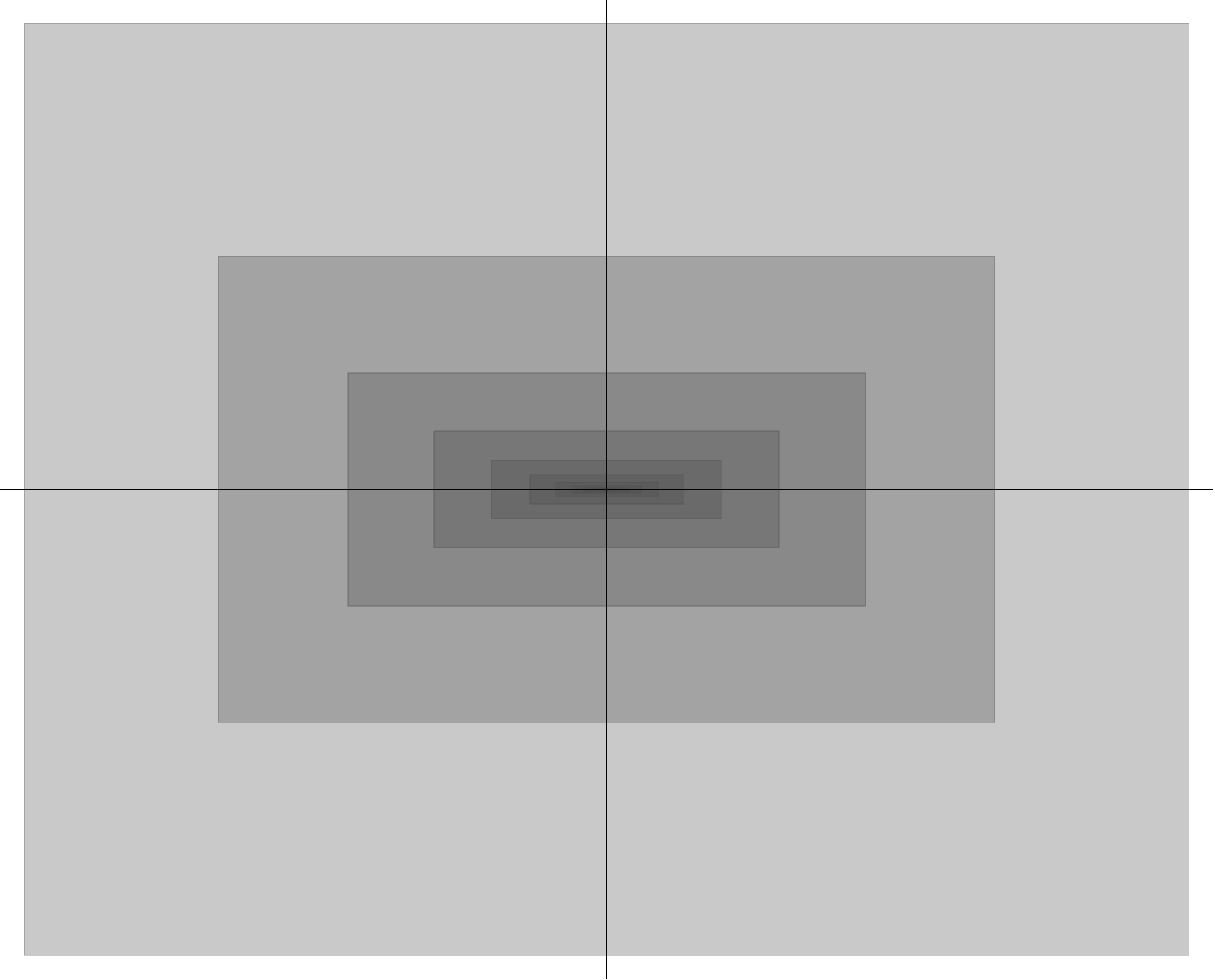}
\caption{In both pictures the set $G$ corresponds to the biggest rectangle. The smaller and smaller rectangles are its preimages by $DT_\zeta^{p}$. A transparency effect was used so that the superpositions can be detected and help us to identify the sets $G^{(j)}$. In  the left case, $0\notin G$ and $m=3$. Observe that $G^{(1)}$ is the union of the lighter regions, $G^{(2)}$ the medium grey parts and $G^{(3)}$ the union of the darker rectangles. In the right case, the sets $G^{(j)}$ correspond to the rectangular annuli that get darker and darker as $j$ increases.}
\label{fig:decomposition}
\end{figure}

Using the stability of the Poisson distribution we can write that 
$$
N^*(E)=N^*(J\times G)=\sum_{i=1}^M \sum_{\ell=0}^\infty \I_{DT_\zeta^{\ell p}(U_i)\in G},
$$
where $M$ is a Poisson r.v. of mean $\theta (b-a) |B_{\gamma^+}(\zeta)|$, $(U_i)_{i\in\N}$ is an  iid sequence with $U_i\sim \U_{B_{\gamma^+}(\zeta)}$ and independent of $M$. Observing that
$U_i\in G^{(j)}\iff \sum_{\ell=0}^\infty \I_{DT_\zeta^{\ell p}(U_i)\in G}=j$, for all $j\in \N_0$, then using \eqref{eq:return-average-formula} and recalling that $\theta=1-\left| \det \left(DT_\zeta^{-p}\right)\right|$, it follows
\begin{align*}
\E(\tilde N(E))&=\E(M)\E\left(\sum_{\ell=0}^\infty \I_{DT_\zeta^{\ell p}(U_i)\in G}\right)=\theta (b-a) |B_{\gamma^+}(\zeta)|\sum_{j=0}^\infty j\, \frac{|G^{(j)}|}{|B_{\gamma^+}(\zeta)|}\\
%&=\theta (b-a) \tau_2\left[1\frac{|L_1|}{\tau_2}+2\frac{|L_2|}{\tau_2}+\ldots+m\frac{|L_m|}{\tau_2}+m\frac{|L_{m+1}^*|}{\tau_2}+(m+1)\frac{|L_{m+1}|}{\tau_2}\right]\\
&=\theta (b-a)\sum_{i=0}^\infty \left| DT_\zeta^{-ip}(G)\right|=\theta (b-a)\sum_{i=0}^\infty \left| \det \left(DT_\zeta^{-p}\right)\right|^i|G|=(b-a)|G|.
\end{align*}
\end{proof}

\begin{proof}[Proof of Corollary~\ref{cor:dichotomy}]
We start by noting that $\mu(N_n^1(J)=k)=\mu(N_n(J\times [0,\tau))=k)\sim \mu(N_n^*(J\times B_{(\tau/|B_1(\zeta)|)^{1/d}}(\zeta))=k)$, which implies that $\p(N^1(J)=k)=\p(N(J\times [0,\tau))=k)= \p(N^*(J\times B_{(\tau/|B_1(\zeta)|)^{1/d}}(\zeta))=k)$. 

Observe that if $j=(\tau/|B_1(\zeta)|)^{1/d}\in\N$, then letting $Z=\min\{T_{1,1}, T_{2,1}, \ldots, T_{j,1}\}$, we have $Z\sim\exp(\sum_{i=1}^j\theta\leb(B_i(\zeta)\setminus B_{i-1}(\zeta)))=\exp(\theta\tau)$. Moreover, if we define the r.v. $U:=U_{i,1}$ whenever $Z=T_{i,1}$ for each $i=1,\ldots,j$, then $U\sim \U(B_{(\tau/|B_1(\zeta)|)^{1/d}}(\zeta))$, wich means that $U$ is a random point chosen on $B_{(\tau/|B_1(\zeta)|)^{1/d}}(\zeta)$ according to normalised Lebesgue measure, \ie $\leb/\tau$. Now observe that when $(\tau/|B_1(\zeta)|)^{1/d}\notin\N$, by letting $j=\lfloor(\tau/|B_1(\zeta)|)^{1/d}\rfloor+1$ and setting $T_{j,1}\sim\exp(\theta |B_{(\tau/|B_1(\zeta)|)^{1/d}}(\zeta)\setminus B_{j-1}(\zeta)|)$, $U_{i,j}\sim \U(B_{(\tau/|B_1(\zeta)|)^{1/d}}(\zeta)\setminus B_{j-1}(\zeta))$, we still obtain that $Z\sim\exp(\theta\tau)$ and $U\sim \U(B_{(\tau/|B_1(\zeta)|)^{1/d}}(\zeta))$.

Hence, we only need to check that the higher dimensional processes provides the same multiplicity distribution as expected for the limit process $\tilde N^1$. Recall that $\theta=1-|\det(DT_\zeta^p)|^{-1}$. %Veja-se Exemplo 33.5 da página 432 do Billingsley
\begin{align*}
\pi(k)&=\p(D_i=k)=\p(N^1(\{t\})=k| \,N^1(\{t\})>0)\\
&=\p(N^*(\{t\}\times B_{(\tau/|B_1(\zeta)|)^{1/d}}(\zeta))=k| \,N^1(\{t\}\times B_{(\tau/|B_1(\zeta)|)^{1/d}}(\zeta))>0)\\
&=\p(U\in DT_\zeta^{-(k-1)p}(B_{(\tau/|B_1(\zeta)|)^{1/d}}(\zeta))\setminus DT_\zeta^{-kp}(B_{(\tau/|B_1(\zeta)|)^{1/d}}(\zeta)))\\
%&=\left|DT_\zeta^{-(k-1)p}(B_{(\tau/|B_1(\zeta)|)^{1/d}}(\zeta))\setminus DT_\zeta^{-kp}(B_{(\tau/|B_1(\zeta)|)^{1/d}}(\zeta))\right|\\
&=\left|DT_\zeta^{-(k-1)p}(B_{(\tau/|B_1(\zeta)|)^{1/d}}(\zeta))\right|/\tau-\left|DT_\zeta^{-kp}(B_{(\tau/|B_1(\zeta)|)^{1/d}}(\zeta))\right|/\tau\\
&=(1-\theta)^{k-1}\theta\left|B_{(\tau/|B_1(\zeta)|)^{1/d}}(\zeta)\right|/\tau=(1-\theta)^{k-1}\theta .
\end{align*}
\end{proof}

\begin{proof}[Proof of Corollary~\ref{cor:two-dimensional}]
The proof follows from  noting that 
$$\mu(N_n(J\times (\tau_1,\tau_2])=k)\sim \mu(N_n^*(J\times (B_{(\tau_2/|B_1(\zeta)|)^{1/d}}(\zeta)\setminus B_{(\tau_1/|B_1(\zeta)|)^{1/d}}(\zeta)))=k),$$ which implies that 
$$\p(N^\dag(J\times (\tau_1,\tau_2])=k)= \p(N^*(J\times (B_{(\tau_2/|B_1(\zeta)|)^{1/d}}(\zeta)\setminus B_{(\tau_1/|B_1(\zeta)|)^{1/d}}(\zeta)))=k).$$ 
From this observation, it follows that the outer measure $\nu$ in this case is given by:
\begin{equation}
\label{eq:outer-measure-general-2d}
\nu(A_k)=\nu(\cup_{s=1}^{\varsigma_k}(\tau_{k,2s-1}, \tau_{k,2s}])=\left|\tilde A_k \setminus\bigcup_{i=1}^{+\infty} DT_\zeta^{-i \,p} (\tilde A_k)\right|,
\end{equation}
where $\tilde A_k=\bigcup_{s=1}^{\varsigma_k}(B_{(\tau_{k,2s}/|B_1(\zeta)|)^{1/d}}(\zeta)\setminus B_{(\tau_{k,2s-1}/|B_1(\zeta)|)^{1/d}}(\zeta))$. The fact that $N^\dag$ corresponds to this outer measure $\nu$ follows as in the proof of similar statements for $\tilde N$ and $N^*$ in the proof of Theorem~\ref{thm:convergence-hyperbolic}.
\end{proof}

\begin{remark}
\label{rem:Mori-Hsing}
The Mori-Hsing characterisation \eqref{Mori-Hsing}, for the limiting 2-dimensional process of $N_n$, follows once four conditions labelled as (A1)--(A4), in \cite{H87}, are verified. The main result in \cite[Theorem~4.5]{H87} essentially asserts that under condition $\Delta$ (similar to Leadbetter's $D(u_n)$) the limiting process satisfies these four conditions and therefore can be written as in  \eqref{Mori-Hsing}.  In this context, we have that,  as in the proof of \cite[Lemma~4.3]{H87}, (A1) follows from stationarity and (A3) is easily proved. Under conditions $\D^*_q(u_n)$ and $\D'_q(u_n)$, property (A4) follows from Theorem~\ref{thm:indep-increments}. Using the same theorem, if the outer measure $\nu$ satisfies an homogeneity property, \ie $\nu(\beta A)=\beta\nu(a)$ for all $\beta>0$, then property (A2) is also easily verified. A close look to \eqref{eq:outer-measure-ex-1}, \eqref{eq:outer-measure-general-2d} and \eqref{eq:outer-measure-eventually-aperiodic} reveals that in all these cases $\nu$ has the homogeneity property  and consequently (A2) holds and therefore the corresponding limit processes $\tilde N$, $N^\dag$ and $\hat N$ are of Mori-Hsing type.

\end{remark}

\subsection{Worked out examples}\label{subsec:examples}

In this section, we give some concrete examples to portray the applicability of the previous results and also how we can easily extend them to situations when the maximal set is made out of a discontinuity point $\zeta$ or multiple correlated points, in order to illustrate the full power of the theory.  

\begin{example}
\label{ex:2d}
Consider a two-dimensional uniformly expanding map, for definiteness take $f:S^1\times S^1\to S^1\times S^1$, given by $f(x,y)=(2x\,\mbox{mod}\,1, 3y\,\mbox{mod}1)$. This is a product map for which the 2-dimensional Lebesgue measure is invariant. Note that for all $\zeta$, we have $$Df_\zeta=\begin{bmatrix}2&0\\0&3
\end{bmatrix}.
$$ Clearly, Theorem~\ref{thm:convergence-hyperbolic} applies and, consequently, when $\zeta$ is non-periodic, $N_n$ and $N_n^*$ converge to a 2-dimensional Poisson process and a 3-dimensional Poisson process, respectively, with intensity measure given by Lebesgue measure. (Note that in this case, the multi-dimensional process $N_n^*$ is defined on the 3-dimensional space: $[0,\infty)\times \R^2$). Observe that for periodic points $\zeta$, Theorem~\ref{thm:convergence-hyperbolic} only gives the convergence of $N_n^*$ and not that of $N_n$ since the expansion rates are distinct, namely, 2 in the horizontal direction and 3 in the vertical direction. To obtain the limit for $N_n$, we need to use Corollary~\ref{cor:two-dimensional}. In order to provide a concrete example we assume that $\zeta=0$ and observe that the limiting process $N^\dag$, in this case, can be written as:
\begin{equation}
\label{eq:N-dag-example}
N^\dag=\sum_{i,j=1}^\infty\sum_{\ell=0}^{\infty} \delta_{\left(T_{i,j},\,\left|B_{\dist\left(\sqrt{U_{i,j}/\pi}Df_\zeta^{\ell }( \cos\Theta_{i,j},\sin\Theta_{i,j}),0\right)}(\zeta)\right|\right)},
\end{equation}
where  $T_{i,j}$ is given by \eqref{T-matrix}, for $\bar T_{i,j}\stackrel[]{D}{\sim}\text{Exp}\left(\theta\right)$, with $\theta=1-|\det^{-1}(DT^p_\zeta)|$; $U_{i,j}\stackrel[]{D}{\sim} \mathcal U_{(i-1,i]}$; $\Theta_{i,j}\sim \U_{[0,2\pi)}$ and $(\bar T_{i,j})_{i,j\in\N}$, $(U_{i,j})_{i,j\in\N}$ and $(\Theta_{i,j})_{i,j\in\N}$ are mutually independent. Observe that for $\ell=0$, we have $\left|B_{\dist\left(\sqrt{U_{i,j}/\pi}Df_\zeta^{\ell }( \cos\Theta_{i,j},\sin\Theta_{i,j}),0\right)}(\zeta)\right|=U_{i,j}$, which is coherent with the Mori-Hsing representation given in \eqref{Mori-Hsing}.

\end{example}

In order to illustrate different types of limiting point processes, described by different outer measures $\nu$, we consider different types of maximal sets. The situation that follows appears naturally, for example, when the maximal set is built up by two consecutive points of the orbit of some non-periodic point $\zeta$ or when the maximal set is reduced to a single point $\zeta$ which is a discontinuity point of the map. Since this last situation was particularly treated in \cite{AFV15}, for comparison purposes, we will use it to exhibit the occurrence of different limiting point processes. 

\begin{example}
\label{ex:discontinuous}
We consider the case of a singly returning eventually aperiodic point $\zeta$ studied in \cite[Section~3.3]{AFV15}. Here, we make some simplifications in order to make the presentation lighter but our result implies the convergence of the 1-dimensional empirical REPP  covered by statement 2(b) of  \cite[Proposition~3.5]{AFV15}.  Let $f$ be a topologically mixing Rychlik map as described in \cite{R83} or \cite[Section~3.2.1]{AFV15}. For simplicity and definiteness, we assume that we are using the usual metric and the invariant measure $\mu=\leb$. Let $\zeta$ be a discontinuity point of $f$, where the right and left derivatives are defined and take values $\beta^+$ and $\beta^-$, respectively. We further assume that $f(\zeta^+)=\zeta^-$ and $f^j(\zeta^-)\neq\zeta$, for all $j\in\N$ (see Figure~\ref{fig:aperiodic-point}). Note that the points to the left of $\zeta$ take a long time to return near $\zeta$. For the points to the right of $\zeta$, a part corresponding to the points of the thinner line segment also take a long time to return close to $\zeta$, while the points of the thicker line segment return immediately after one iteration but they return to the left of $\zeta$, which means that after that, they will also take a long time to return again to the small neighbourhood of $\zeta$ depicted in the picture.   
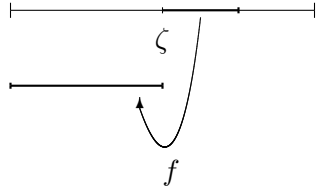
\begin{figure}
\setlength{\unitlength}{1cm}
\begin{picture}(2,2)(-1,-1)
% the set $U_n$
\put(-2,0){\line(1,0){4}}
%left end point
\put(-2,0){\line(0,1){0.1}}
\put(-2,0){\line(0,-1){0.1}}
%right end point
\put(2,0){\line(0,1){0.1}}
\put(2,0){\line(0,-1){0.1}}
% centre
\put(0,0){\line(0,1){0.04}}
\put(0,0){\line(0,-1){0.04}}
\put(-0.1,-0.5){$\zeta$}
\thicklines
\put(0,0){\line(1,0){1}}
\put(1,0){\line(0,1){0.04}}
\put(1,0){\line(0,-1){0.04}}
%after iteration
\put(0,-1){\line(-1,0){2}}
\put(0,-1){\line(0,1){0.04}}
\put(0,-1){\line(0,-1){0.04}}
\put(-2,-1){\line(0,1){0.04}}
\put(-2,-1){\line(0,-1){0.04}}
\thinlines
%\put(-2,-1){\line(-1,0){2}}
%\put(-4,-1){\line(0,1){0.1}}
%\put(-4,-1){\line(0,-1){0.1}}

%iteration arrow
\qbezier(0.5,-0.1)(0.2,-2.75)(-0.3,-1.3)
\put(-0.3,-1.25){\vector(0,0){0.1}}
%\put(0.25,-0.4){\line(0,-1){1}}
%\put(0,-1.4){\oval(0.5,0.5)[b]}
%\put(-0.25,-1.5){\vector(0,1){0.25}}
\put(0,-2.25){$f$}
%\put(-3,-3){\small Fig.2 - Singly returning, eventually aperiodic $\zeta$}
\end{picture}
\vspace{1cm}
\caption{Singly returning eventually aperiodic point}\label{fig:aperiodic-point}
\end{figure}
The sets $A_{n,k}$, in this case, correspond to the union of $2\varsigma_k$ disjoint intervals which are placed symmetrically to the left and right side of $\zeta$. When computing $\A_{n,k}$ we note that the intervals of the left (which correspond to half of the measure of $A_{nk,}$) have no exclusions, while the ones on the right may have exclusions because some of the inner intervals (closer to $\zeta$) may hit the left side intervals. It is not hard to verify that in this case, we have:
\begin{equation}
\label{eq:outer-measure-eventually-aperiodic}
\nu(A_k)=\frac12|A_k|+\frac12|A_k\setminus (\beta^+)^{-1}A_k|.
\end{equation} 
In this situation the limiting process can be written as:
\begin{equation}
\label{eq:repelling-limit-REPP-2d-special}
 \hat N=\sum_{i,j=1}^\infty\sum_{\ell=0}^{1}\left(\delta_{(T_{i,j},\, U_{i,j})}\I_{\{Z_{i,j}=0\}}+\delta_{(T_{i,j},\,(\beta^+)^{\ell }\, U_{i,j})}\I_{\{Z_{i,j}=1\}}\right),
\end{equation}
where  $T_{i,j}$ is given by \eqref{T-matrix}, for $\bar T_{i,j}\stackrel[]{D}{\sim}\text{Exp}(\theta)$, with $\theta=1-(\beta^+)^{-1}/2$, $U_{i,j}\stackrel[]{D}{\sim} \mathcal U_{(i-1,i]}$, $\p(Z_{i,j}=0)=\frac{2\theta-1}{2\theta}$, $\p(Z_{i,j}=1)=\frac{1}{2\theta}$ and $(\bar T_{i,j})_{i,j\in\N}$, $(U_{i,j})_{i,j\in\N}$ and $(Z_{i,j})_{i,j\in\N}$ are mutually independent. Let $\gamma:=\p(Z_{i,j}=1)$. We are only left to check that $\hat N$ is compatible with $\nu$ given by \eqref{eq:outer-measure-eventually-aperiodic}. 
\begin{align*}
\p\Big(\hat N\left(J_k\times A_k\right)=0\Big)&=\e^{-\theta\tau_{k,2\varsigma_k}|J_k|}\cdot \e^{\theta\tau_{k,2\varsigma_k}|J_k|\left((1-\gamma)\frac{\tau_{k,2\varsigma_k}-|A_k|}{\tau_{k,2\varsigma_k}}+\gamma\frac{\tau_{k,2\varsigma_k}-|A_k\cup (\beta^+)^{-1}A_k|}{\tau_{k,2\varsigma_k}}\right)}\\
&=\exp\left\{-\theta(1-\gamma)|J_k||A_k|-\theta\gamma|J_k||A_k\cup (\beta^+)^{-1}A_k|\right\}
\end{align*}
Finally, recalling that $\gamma=1/2\theta$, we have
\begin{align*}
\theta&(1-\gamma)|A_k|+\theta\gamma|A_k\cup (\beta^+)^{-1}A_k|=\frac{2\theta-1}{2}|A_k|+\frac12|A_k\setminus (\beta^+)^{-1}A_k| +\frac12|(\beta^+)^{-1}A_k|\\
&=\frac12(1-(\beta^+)^{-1})|A_k|+\frac12|A_k\setminus (\beta^+)^{-1}A_k| +\frac12|(\beta^+)^{-1}A_k|=\frac12|A_k|+\frac12|A_k\setminus (\beta^+)^{-1}A_k|.
%&=|A_k\setminus\cup_{i=1}^\infty \alpha^{-i\cdot d}A_k|+|\cup_{i=1}^\infty \alpha^{-i\cdot d}A_k|-|\cup_{i=1}^\infty \alpha^{-i\cdot d}A_k|=|A_k\setminus\cup_{i=1}^\infty \alpha^{-i\cdot d}A_k|
\end{align*}
We focus now on the second criterion of Kallenberg. Let  $E=J\times G$. The same argument as in \eqref{eq:Kallenberg2} gives us that $\lim_{n\to\infty}\E(N_n(E))=|E|$. Recall that $N$ stands for the 2-dimensional homogeneous Poisson process given by \eqref{Poisson-2d} (such that $\E(N(E))=\leb(E)$).  Then 
\begin{align*}
\E(\tilde N(E))&=(1-\gamma)\E\left(\sum_{i,j=1}^\infty\delta_{(\theta^{-1}T_{i,j}, U_{i,j})}(E)\right)+\gamma\E\left(\sum_{i,j=1}^\infty\sum_{\ell=0}^{1}\delta_{(\theta^{-1}T_{i,j},(\beta^+)^{\ell }\,U_{i,j})}(E)\right)\\
&=(1-\gamma)\E\left(\sum_{i,j=1}^\infty\delta_{(T_{i,j}, U_{i,j})}(\theta J\times G)\right)+\beta\sum_{\ell=0}^{1}\E\left(\sum_{i,j=1}^\infty\delta_{(T_{i,j},U_{i,j})}(\theta J\times (\beta^+)^{-\ell }G)\right)\\
&=(1-\gamma)\E(N(\theta J\times G))+\gamma\E(N(\theta J\times G))+\gamma\E(N(\theta J\times (\beta^+)^{-1}G))\\
&=|J\times G|((1-\gamma)\theta+\gamma\theta+\gamma\theta(\beta^+)^{-1})=|E|.
\end{align*}

\end{example}

Next, we consider an example where the maximal set has two non-periodic points belonging to the same orbit as in \cite[Section~4]{AFFR16}. There are a couple of novelties here, when compared to the previous cases. Namely, in certain occasions, once a point is marked, the second point on the vertical pile is marked below and not above the first one. Moreover, the correction to the frequency of time occurrences is not the EI itself, anymore.

\begin{example}
\label{ex:exception}
Let $f:[0,1]\to[0,1]$, be such that $f(x)=3x\,\text{mod1}$. Also set $\varphi:[0,1]\to\R$, such that $\varphi(x)=\max\{0,1-100|x-\pi/16|\}+\max\{0,1-10|x-\pi/16|\}$. Observe that $\varphi$ is maximised at $\pi/16$ and $f(\pi/16)=3\pi/16$, which are not periodic points. Observe that 
$$
\{X_0>u\}=B_{\frac{1-u}{100}}\left(\pi/{16}\right)\cup B_{\frac{1-u}{10}}\left(3\pi/{16}\right)=:\mathcal V_1(u)\cup \mathcal V_2(u).
$$
Note that $|\mathcal V_1(u)|/|\{X_0>u\}|=1/11$ and $|\mathcal V_2(u)|/|\{X_0>u\}|=10/11$.
Since we want \eqref{un} to hold, we take $|\{X_0>u_n\}|=\frac{2(1-u_n)}{100}+\frac{2(1-u_n)}{10}=\tau/n$ and hence $u_n(\tau)=1-\frac{100}{11}\frac\tau{2n}$, which in turn implies that $u_n^{-1}(z)= 2n\frac{11}{100}(1-z)$.

Given $A_k$ as in \eqref{eq:rectangles}, we see that $A_{n,k}$ corresponds to $2\varsigma_k$ intervals symmetrically placed in each $\mathcal V_i(u_n(\tau_{k,2\varsigma_k}))$, $i=1,2$. Our particular choice of maximal points implies that for all $j\in\N$ we have $f^j(\mathcal V_2(u_n(\tau_{k,2\varsigma_k})))\cap \{X_0>u_n\}=\emptyset$, for $n$ sufficiently large, and $f(\mathcal V_1(u_n(\tau_{k,2\varsigma_k})))\cap \mathcal V_2(u_n(\tau_{k,2\varsigma_k}))\neq \emptyset$, for all $n\in\N$. 
In particular, the EI is given by $\theta=\lim_{n\to\infty}\frac{|\mathcal V_2(u_n(\tau))|}{|\{X_0>u_n(\tau)\}|}=\frac{10}{11}.$
Now, we observe that
\begin{align*}
&A_{n,k}^{(1)}=A_{n,k}\cap \mathcal V_2(u_n(\tau_{k,2\varsigma_k}))\cup \left(A_{n,k}\cap \mathcal V_1(u_n(\tau_{k,2\varsigma_k}))\right)\setminus f^{-1}\left(A_{n,k}\cap \mathcal V_2(u_n(\tau_{k,2\varsigma_k}))\right)\\
&A_{n,k}\cap \mathcal V_2(u_n(\tau_{k,2\varsigma_k}))=\frac{10}3\left(f(A_{n,k}\cap \mathcal V_1(u_n(\tau_{k,2\varsigma_k})))-3\pi/6\right)+3\pi/6.
\end{align*}
It follows that 
$$\nu(A_k)=\lim_{n\to\infty}n|A_{n,k}^{(1)}|=\frac{10}{11}|A_k|+\frac{1}{11}\left|A_k\setminus \frac{10}{3} A_k\right|.$$
In this situation the limiting process can be written as:
\begin{equation}
\label{eq:repelling-limit-REPP-2d-special2}
\doublehat N=\sum_{i,j=1}^\infty\sum_{\ell=0}^{1}\left(\delta_{(T_{i,j},\, U_{i,j})}\I_{\{Z_{i,j}=0\}}+\delta_{(T_{i,j},\,(3/10)^{\ell }\, U_{i,j})}\I_{\{Z_{i,j}=1\}}\right),
\end{equation}
where  $T_{i,j}$ is given by \eqref{T-matrix}, for $\bar T_{i,j}\stackrel[]{D}{\sim}\text{Exp}(\Theta)$, with $\Theta=23/33$, $U_{i,j}\stackrel[]{D}{\sim} \mathcal U_{(i-1,i]}$, $\p(Z_{i,j}=0)=20/23$, $\p(Z_{i,j}=1)=3/23$ and $(\bar T_{i,j})_{i,j\in\N}$, $(U_{i,j})_{i,j\in\N}$ and $(Z_{i,j})_{i,j\in\N}$ are mutually independent. Let $\beta:=\p(Z_{i,j}=1)$. We are only left to check that $\hat N$ is compatible with $\nu$ given above. 

In order to understand the formula above for the limiting $\doublehat N$ point process, we start by observing the following. Let $u<1$ be close to $1$ and $x\in[0,1]$ be close to $\pi/16$ such that $\varphi(x)=u$, which means that $|x-\pi/16|=(1-u)/100$. Then $|f(x)-3\pi/16|=3(1-u)/100$ and therefore $\varphi(f(x))=1-3/10(1-u)$. Consequently, we obtain $u_n^{-1}(\varphi(x))=2n(1-u)11/100$ and $u_n^{-1}(\varphi(f(x)))=3/10 u_n^{-1}(\varphi(x))$.

Consider the interval $[0,\tau)$. A point of mass of $\doublehat N$ with a vertical coordinate in this range corresponds to an exceedance of $u_n(\tau)$. The case $Z_{i,j}=0$ corresponds to an entrance in $\mathcal V_2(u_n(\tau))$, which means the cluster has size 1 and therefore only one point appears on the same vertical line. The case $Z_{i,j}=1$ corresponds to an entrance in $\mathcal V_1(u_n(\tau))$, which means that a more severe exceedance will occur immediately after, which explains the appearance of another point on the same vertical line, but this time at a distance from the horizontal axis equal to $3/10$ of distance of the first point marked in that vertical line.

In order to show the correct choice for $\Theta$, we start by the second criterion of Kallenberg. Let  $E=J\times G$. The same argument as in \eqref{eq:Kallenberg2} gives us that $\lim_{n\to\infty}\E(N_n(E))=|E|$. Recall that $N$ stands for the 2-dimensional homogeneous Poisson process given by \eqref{Poisson-2d} (such that $\E(N(E))=\leb(E)$).  Then 
\begin{align*}
\E(\tilde N(E))&=(1-\beta)\E\left(\sum_{i,j=1}^\infty\delta_{(\Theta^{-1}T_{i,j}, U_{i,j})}(E)\right)+\beta\E\left(\sum_{i,j=1}^\infty\sum_{\ell=0}^{1}\delta_{(\Theta^{-1}T_{i,j},(3/10)^{\ell }\,U_{i,j})}(E)\right)\\
&=(1-\beta)\E\left(\sum_{i,j=1}^\infty\delta_{(T_{i,j}, U_{i,j})}(\Theta J\times G)\right)+\beta\sum_{\ell=0}^{1}\E\left(\sum_{i,j=1}^\infty\delta_{(T_{i,j},U_{i,j})}(\Theta J\times (10/3)^{\ell }G)\right)\\
&=(1-\beta)\E(N(\Theta J\times G))+\beta\E(N(\Theta J\times G))+\beta\E(N(\Theta J\times 10/3G))\\
&=|J\times G|((1-\beta)\Theta+\beta\Theta+\beta\Theta10/3)=|E|.
\end{align*}
 We now turn to the first criterion.
\begin{align*}
\p\Big(\doublehat N\left(J_k\times A_k\right)=0\Big)&=\e^{-\Theta\frac{10}3\tau_{k,2\varsigma_k}|J_k|}\cdot \e^{\Theta\frac{10}3\tau_{k,2\varsigma_k}|J_k|\left((1-\beta)\frac{\frac{10}3\tau_{k,2\varsigma_k}-|A_k|}{\frac{10}3\tau_{k,2\varsigma_k}}+\beta\frac{\frac{10}3\tau_{k,2\varsigma_k}-|A_k\cup \frac{10}3A_k|}{\frac{10}3\tau_{k,2\varsigma_k}}\right)}\\
&=\exp\left\{-\Theta(1-\beta)|J_k||A_k|-\Theta\beta|J_k||A_k\cup \frac{10}3A_k|\right\}
\end{align*}
Finally, recalling that $\gamma=1/2\theta$, we have
\begin{align*}
\Theta&(1-\beta)|A_k|+\Theta\beta|A_k\cup \frac{10}3A_k|=\Theta(1-\beta)|A_k|+\Theta\beta |A_k\setminus \frac{10}3A_k| +\Theta\beta|\frac{10}3A_k|\\
&=\Theta((1-\beta)+\frac{10}3\beta)|A_k|+\Theta\beta|A_k\setminus \frac{10}3A_k| =\frac{10}{11}|A_k|+\frac1{11}|A_k\setminus \frac{10}3A_k|.
%&=|A_k\setminus\cup_{i=1}^\infty \alpha^{-i\cdot d}A_k|+|\cup_{i=1}^\infty \alpha^{-i\cdot d}A_k|-|\cup_{i=1}^\infty \alpha^{-i\cdot d}A_k|=|A_k\setminus\cup_{i=1}^\infty \alpha^{-i\cdot d}A_k|
\end{align*}

\end{example}

\section{Records, convergence to extremal processes and other consequences}

As in \eqref{eq:SP-def}, let $X_0, X_1, \ldots$ be a stationary stochastic process, where each r.v. $X_i:\X\to\R$ is defined on the measure space $(\X,\mathcal B,\p)$ and $\p$ is $T$-invariant. Define the partial maximum
\[M_n:=\max\{X_0,\ldots,X_{n-1}\}.\]
Similarly to central limit laws for partial sums, we will assume the existence of normalising sequences $(a_n)_{n\in\N}\subset\R^+$ and $(b_n)_{n\in\N}\subset\R$ such that
\begin{equation}\label{eq.Mn-limit}
\p(a_n(M_n-b_n)\leq y)\to G(y),
\end{equation}
for some non-degenerate distribution function (d.f.) $G$, as $n\to\infty$.

Beyond the distributional limit established in \eqref{eq.Mn-limit}, here we consider the continuous time 
process $\{Y_n(t):t\geq 0\}$ defined by
\begin{equation}\label{eq.Y-process}
Y_n(t):=a_n(M_{\lfloor nt\rfloor+1}-b_n)\\
\end{equation}

For each $n\geq 1$, $Y_n(t)$ is a random graph with values in the Skorokhod space $\mathbb{D}((0,\infty))$ of right continuous functions, with existence of limits to the left (cadlag functions). Under suitable hypotheses on the dynamical system $(\X,T,\p)$, the existence of a non-degenerate limit process $Y(t)$ so that $Y_n\Rightarrow Y$ in $\mathbb{D}((0,\infty))$ with respect to the Skorokhod's $J_1$ topology has been proved in \cite{HT17}. The limit process $Y(t)$ will be the so called \emph{extremal process} which we now define.

\subsection{Extremal processes and weak convergence.}
\label{subsec:extremal-processes}

Consider a general probability space $(\Omega,\mathcal{B},\p)$, where $\mathcal{B}$ is the $\sigma$-algebra of sets in the sample space $\Omega$. If $X:\Omega\to\mathbb{R}$ is a random variable, we let $F(u):=\p(X\leq u)$, and define finite dimensional distributions:
\begin{equation}
F_{t_1,\ldots,t_k}\left(u_1,\ldots,u_k\right)=
F^{t_1}\left(\bigwedge_{i=1}^{k}\{u_i\}\right)
F^{t_2-t_1}\left(\bigwedge_{i=2}^{k}\{u_i\}\right)\cdots 
F^{t_k-t_{k-1}}(u_k),
\end{equation}

with $t_1<t_2<\cdots<t_k$, and $\wedge$ denoting the minimum operation. Suppose that $Y_{F}(t)$ is a stochastic process with these finite dimensional distributions, i.e.
\begin{equation}
\p(Y_F(t_1)\leq u_1,\ldots, Y_F(t_k)\leq u_k)=F_{t_1,\ldots,t_k}(u_1,\ldots,u_k).
\end{equation}

By the Kolmogorov extension theorem such a process exists and is called an \emph{extremal-$F$ process}. A version can be taken in $\mathbb{D}((0,\infty))$, i.e. continuous to the right with left hand limits. As we mentioned before, for certain chaotic dynamical systems, the process $Y_n(t)$ in \eqref{eq.Y-process} converges (weakly) to an extremal-$G$ process $Y_G(t)$.

%\subsection{Extreme Value Laws}
%\label{subsec:EVLs}

\begin{definition}
We say that we have an \emph{Extreme Value Law} (EVL) for $M_n$ if there is a non-degenerate d.f. $H:\R_0^+\to[0,1]$ with $H(0)=0$ and, for every $\tau>0$, there exists a sequence of thresholds $u_n=u_n(\tau)$, $n=1,2,\ldots$, satisfying equation \eqref{un} and for which the following holds:
\begin{equation}
\label{eq:EVL-law}
\p(M_n\leq u_n)\to \bar H(\tau),\;\mbox{ as $n\to\infty$.}
\end{equation}
where $\bar H(\tau):=1-H(\tau)$ and the convergence is meant at the continuity points of $H(\tau)$.
\end{definition}

In this context, we now consider the continuous time 
process $\{Z_n(t):t\geq 0\}$ defined by
\begin{equation}
\label{eq.Z-process}
Z_n(t):=u_n^{-1}(M_{\lfloor nt\rfloor+1})
\end{equation}

For each $n\geq 1$, $Z_n(t)$ is a random graph with values in the Skorokhod space $\mathbb{D}((0,\infty))$. Under suitable hypotheses on $(\X,T,\p)$ we prove the existence of a non-degenerate limit process $Z_H(t)$ so that $Z_n(t)\overset{d}{\to}Z_H(t)$ in $\mathbb{D}((0,\infty))$ with respect to the Skorokhod's $M_1$ topology, where $Z_H(t)$ is a stochastic process such that
\begin{equation}
\p(Z_H(t_1)\geq y_1,\ldots Z_H(t_k)\geq y_k)=\bar H^{t_1}\left(\bigvee_{i=1}^k\{y_i\}\right)\bar H^{t_2-t_1}\left(\bigvee_{i=2}^k\{y_i\}\right)\cdots\bar H^{t_k-t_{k-1}}(y_k),
\end{equation}

with $0\leq t_1<t_2<\cdots<t_k$, and $\vee$ denoting the maximum operation. Furthermore, $\bar H(\tau)=\e^{-\theta\tau}$, where $\theta\in(0,1]$ is the parameter called \emph{extremal index} whose existence was previously assumed.

In particular, if there are normalising sequences $(a_n)_{n\in\N}\subset\R^+$ and $(b_n)_{n\in\N}\subset\R$ such that we can write $u_n=y/a_n+b_n$ and
\[n\p(X_0>u_n)=n\p(a_n(X_0-b_n)>y)\to\tau\]
with $\tau=f(y)$ for some homeomorphism $f$, then, as in \eqref{eq.Mn-limit},
\[\p(M_n\leq u_n)=\p(a_n(M_n-b_n)\leq y)\to G(y)\]
where $G=\bar H\circ f$. Hence,
\begin{align*}
u_n(t)&=a_n^{-1}f^{-1}(t)+b_n\\
u_n^{-1}(z)&=f(a_n(z-b_n))\\
Z_n(t)&=f(Y_n(t))\\
Y_n(t)&=f^{-1}(Z_n(t))
\end{align*}

\begin{remark}
In the iid setting, the convergence in \eqref{un} is equivalent to
\[\p(M_n\leq u_n)\to\e^{-\tau},\;\mbox{ as $n\to\infty$,}\]
hence $\theta=1$. Depending on the type of limit law that applies, $f(y)$ is of one of the following three types: $f_1(y)=\e^{-y}$ for $y\in{\mathbb R}$, $f_2(y)=y^{-\alpha}$ for $y>0$, and $f_3(y)=(-y)^{\alpha}$ for $y\leq 0$.
\end{remark}

We will show that for certain chaotic dynamical systems the process $Z_n(t)$ in \eqref{eq.Z-process} converges (weakly) to $Z_H(t)$ in $\mathbb{D}([0,\infty))$ endowed with the Skorokhod's $M_1$ topology. This topology allows a function $g_{1}$ with a jump at $t$ to be approximated arbitrarily well by some continuous $g_{2}$ (with large slope near $t$). For any $g\in\mathbb{D}([0,T])$, let $\Gamma(g):=\{(t,x)\in\lbrack0,T]\times\mathbb{R}:x\in[g(t^{-})\wedge g(t),g(t^{-})\vee g(t)]\}$ denote the completed graph of $g$, and let $\Lambda^{\ast}(g)$ be the set of all its parametrizations, that is, all continuous $G=(\lambda,\gamma):[0,T]\rightarrow\Gamma(g)$ such that $t^{\prime}<t$ implies either $\lambda(t^{\prime})<\lambda(t)$ or $\lambda(t^{\prime})=\lambda(t)$ plus $\left\vert\gamma(t)-g(\lambda(t))\right\vert<\left\vert \gamma(t^{\prime})-g(\lambda(t))\right\vert$. Then,
\[d_{M_1,T}(g_{1},g_{2}):=\inf_{G_{i}=(\lambda_{i},\gamma_{i})\in\Lambda^{\ast}(g_{i})}\left\{ \left\Vert\lambda_{1}-\lambda_{2}\right\Vert\vee\left\Vert \gamma_{1}-\gamma_{2}\right\Vert\right\},\]
where $\|\cdot\|$ denotes the uniform norm, gives a metric inducing $M_1$.

On the space $\mathbb{D}([0,\infty))$ the $M_1$ topology is defined by the metric
\[d_{M_1,\infty}(g_{1},g_{2}):=\int_{0}^{\infty}e^{-t}(1\wedge d_{M_1,t}(g_{1},g_{2}))\,dt.\]
Convergence $g_{n}\rightarrow g$ in $(\mathbb{D}([0,\infty)),M_1)$ means that $d_{M_1,T}(g_{n},g)\rightarrow0$ for every continuity point $T$ of $g$.

Let us consider the projections $h_1,h_2:M_p([0,\infty)^2)\to\mathbb{D}([0,\infty))$ where, given the planar point process $m=\sum_{i=1}^{\infty}\delta_{(t_i,y_i)}$, $h_1 m$ is the real valued function defined by
\[h_1 m(t):=\begin{cases}
\inf\{y_i:t_i\leq t\} &\mbox{if } t>\underline{t}\\
\overline{y} &\mbox{if }  t\leq\underline{t}
\end{cases}\]
with $\underline{t}=\inf\{(t_i)_{i=1}^{\infty}\}$ and $\overline{y}=\sup\{\inf\{y_i:t_i\leq t\}:t>\underline{t}\}$, and $h_2 m$ is the real valued function defined by 
\[h_2 m(y):=\begin{cases}
\inf\{t_i:y_i<y\} &\mbox{if } y>\underline{y}\\
\overline{t} &\mbox{if } y\leq\underline{y}
\end{cases}\]
with $\underline{y}=\inf\{(y_i)_{i=1}^{\infty}\}$ and $\overline{t}=\sup\{\inf\{t_i:y_i<y\}:y>\underline{y}\}$.

\begin{theorem}\label{continuity}
$h_1$ and $h_2$ are a.s.\ continuous in the $M_1$ topology of $\mathbb{D}([0,\infty))$ with respect to point processes in $M_p([0,\infty)^2)$, for processes of Mori-Hsing type.% such that the points $(Y_{i,j,\ell})_{\ell\in\N}$ do not acumulate anywhere).
\end{theorem}

\begin{proof}
Consider the projection $h_1$ (the proof is similar for $h_2$). It suffices to show that $h_1$ is continuous in the $M_1$ topology of $\mathbb{D}([0,T])$ at $m\in M_p([0,\infty)^2)$, where $m$ satisfies the following:
\[m(\{0\}\times[0,\infty))=m(\{T\}\times[0,\infty))=m([0,\infty)\times\{0\})=0,\qquad m([0,t]\times(0,x))<\infty\]

for any $0<s<t<T$, $x\in\R$. Note that these properties are satisfied by the type of processes we are considering here. Let $m_n\in M_p([0,\infty)^2)$ and suppose $m_n\overset{v}{\to}m$. We need to show that $d_{M_1,T}(h_1 m_n,h_1 m)\to 0$. Suppose for concreteness that $h_1 m(0)>h_1 m(T)$ and choose $D>h_1 m(0)$ such that $m([0,T]\times\{D\})=0$. For large enough $n$,
\[m_n([0,T]\times(0,D))=m([0,T]\times(0,D))=p,\]
with $1\leq p<\infty$, and there is an enumeration of the points of $m_n$, call it $(t_i^{(n)},y_i^{(n)})_{1\leq i\leq p}$ such that $\lim_{n\to\infty}(t_i^{(n)},y_i^{(n)})=(t_i,y_i)$, $1\leq i\leq p$, where $(t_i,y_i)_{1\leq i\leq p}$ is the analogous enumeration of points of $m$ in $[0,T]\times(0,D)$. Pick $\delta<\frac{1}{2}\min\{|t_i-t_{i'}|:t_i\neq t_{i'}\}\wedge\frac{1}{2}\min\{|y_i-y_{i'}|:y_i\neq y_{i'}\}$ small enough that $\delta$-spheres about the distinct points of the set $\{(t_i,y_i)_{1\leq i\leq p}\}$ are disjoint and in $[0,T]\times(0,D)$. Pick n so large that each sphere contains the same number of points of $m_n$ as of $m$.
Let $(t_{i_k},y_{i_k})_{1\leq k\leq q}$, $q\leq p$, be the points $(t_i,y_i)$ such that $y_i<y_{i'}$ for any $(t_{i'},y_{i'})\neq(t_i,y_i)$ with $t_{i'}\leq t_i$. Additionally, suppose that $0<t_{i_1}<t_{i_2}<\ldots<t_{i_q}<T$ (note that $t_{i_k}\neq t_{i_{k'}}$ for $k\neq k'$) so that the complete graph of $h_1 m$ is just a polygonal line with distinct vertices $(0,y_{i_1}),(t_{i_1},y_{i_1}),(t_{i_2},y_{i_1}),(t_{i_2},y_{i_2}),\ldots,(t_{i_q},y_{i_{q-1}}),(t_{i_q},y_{i_q}),(T,y_{i_q})$ (if $
t_{i_1}=0$ or $t_{i_q}=T$ we simply drop the respective point $(t_{i_1},y_{i_1})$ or $(t_{i_q},y_{i_q})$)

Let $G=(\lambda,\gamma):[0,T]\rightarrow\Gamma(h_1 m)$ be a parametrization of $\Gamma(h_1 m)$ and $G^{(n)}:[0,T]\rightarrow\Gamma(h_1 m_n)$ a parametrization of $\Gamma(h_1 m_n)$. Let $u_k=G^{-1}(t_{i_k},y_{i_k})$ for $1\leq k\leq q$ and $v_k=G^{-1}(t_{i_{k+1}},y_{i_k})$ for $1\leq k<q$. Note that $0<u_1<v_1<u_2<\ldots<v_{q-1}<u_q<T$. Now, we define $h^{(n)}$ as a homeomorphism of $[0,T]$ onto $[0,T]$ by

\begin{itemize}
\item $h^{(n)}(0)=0, \quad h^{(n)}(T)=T$
\item $h^{(n)}(u_k)=(G^{(n)})^{-1}(t_k^{(n)},y_k^{(n)})$ where $(t_k^{(n)},y_k^{(n)})\in\Gamma(h_1 m_n)$ belongs to the sphere of radius $\delta$ about $(t_{i_k},y_{i_k})$ for $1\leq k\leq q$
\item $h^{(n)}(v_k)=(G^{(n)})^{-1}(t_{k'}^{(n)},y_{k'}^{(n)})$ where $(t_{k'}^{(n)},y_{k'}^{(n)})$ belongs to the sphere of radius $\delta$ about $(t_{i_{k+1}},y_{i_k})$ for $1\leq k<q$ (note that $(t_{k+1}^{(n)},y_k^{(n)})$ does not necessarily belong to $\Gamma(h_1 m_n)$)
\item $h^{(n)}$ is linearly interpolated elsewhere on $[0,T]$
\end{itemize}

Then, $G^{(n)}\circ h^{(n)}=(\lambda^{(n)},\gamma^{(n)}):[0,T]\rightarrow\Gamma(h_1 m_n)$ is another parametrization of $\Gamma(h_1 m_n)$, with $\left\Vert\lambda-\lambda^{(n)}\right\Vert<\delta$ and $\left\Vert\gamma-\gamma^{(n)}\right\Vert<\delta$, showing $h_1 m$ and $h_1 m_n$ are at a distance less than $\delta$ on the space $\mathbb{D}([0,T])$ with the metric $d_{M_1,T}$.
\end{proof}

\begin{proposition}\label{Z-process}
If the process $N_n$ converges weakly to $\tilde N$, then the process $Z_n$ converges weakly to $Z_H$ in $\mathbb{D}([0,\infty))$ endowed with the Skorokhod's $M_1$ topology.
\end{proposition}

\begin{proof}
Since $\tilde N$ is a processes of Mori-Hsing type, 
we can use Theorem \ref{continuity} to conclude that $h_1$ is a.s.\ continuous with respect to $\tilde N$ in the $M_1$ topology of $\mathbb{D}([0,\infty))$. Hence, by the CMT we just need to show that $h_1(N_n)=Z_n$ for every $n\in\N$ and $h_1(\tilde N)=Z_H$.

We have, for $t\geq 0$ and $n\in\N$,
\[h_1 N_n(t)=\inf\{u_n^{-1}(X_i):i/n\leq t\}=u_n^{-1}(\sup\{X_i:i\leq n t\})=Z_n(t)\]

Now, let us see that the finite dimensional distributions of $h_1(\tilde N)$ are the same of $Z$. For the unidimensional distribution, with $t,y\geq 0$ and $\bar H(\tau)=\e^{-\theta\tau}$, we have
\[\p(h_1\tilde N(t)\geq y)=\p(\tilde N([0,t]\times[0,y))=0)=\e^{-\theta t y}=\bar H^t(y).\]

For the bidimensional distribution, with $0\leq t_1<t_2$ and $y_1\geq y_2\geq 0$,
\begin{align*}
\p(h_1\tilde N(t_1)\geq y_1,h_1\tilde N(t_2)\geq y_2)&=\p(\tilde N([0,t_1]\times[0,y_1))=0,\tilde N((t_1,t_2]\times[0,y_2))=0)\\
&=\bar H^{t_1}(y_1)\bar H^{t_2-t_1}(y_2).
\end{align*}
In case $0\leq y_1<y_2$,
\[\p(h_1\tilde N(t_1)\geq y_1,h_1\tilde N(t_2)\geq y_2)=\p(h_1\tilde N(t_2)\geq y_2)=\bar H^{t_2}(y_2),\]
so in general
\[\p(h_1\tilde N(t_1)\geq y_1,h_1\tilde N(t_2)\geq y_2)=\bar H^{t_1}(y_1\vee y_2)\bar H^{t_2-t_1}(y_2).\]

Following this pattern we get for the $k$-dimensional distribution
\begin{align*}
\p(h_1\tilde N(t_1)\geq y_1,\ldots,h_1\tilde N(t_k)\geq y_k)&=\bar H^{t_1}\left(\bigvee_{i=1}^k\{y_i\}\right)\bar H^{t_2-t_1}\left(\bigvee_{i=2}^k\{y_i\}\right)\cdots\bar H^{t_k-t_{k-1}}(y_k)\\
&=\p(Z_H(t_1)\geq y_1,\ldots,Z_H(t_k)\geq y_k).
\end{align*}
Hence $h_1(\tilde N)=Z_H$.
\end{proof}

\begin{remark}
If, as before, we assume the existence of normalising sequences $(a_n)_{n\in\N}\subset\R^+$ and $(b_n)_{n\in\N}\subset\R$ such that $n\p(a_n(X_0-b_n)>y)\to f(y)$ and $\p(a_n(M_n-b_n)\leq y)\to G(y)$, with $G=\bar H\circ f$ for some d.f. $G$ and $H$, then $f^{-1}(Z_n(t))=Y_n(t)$ and $f^{-1}(Z_H(t))=Y_G(t)$. Hence, we also have $Y_n\Rightarrow Y_G$ in $\mathbb{D}([0,\infty))$ with respect to the Skorokhod's $M_1$ topology.
\end{remark}

Given the stochastic process $Z_H(t)$ its path inverse is defined by:
\[Z_H^{\leftarrow}(y)=\inf\{t:Z_H(t)<y\},\]
where the domain of $Z_H^{\leftarrow}$ is codomain of $Z$, and $Z_n^{\leftarrow}$ is defined similarly. We have the following result.

\begin{proposition}\label{Z-inv.process}
If the process $N_n$ converges weakly to $\tilde N$, then the process $Z_n^{\leftarrow}$ converges weakly to $Z_H^{\leftarrow}$ in $\mathbb{D}([0,\infty))$ endowed with the Skorokhod's $M_1$ topology.
\end{proposition}

\begin{proof}
Since $h_2$ is again a.s.\ continuous with respect to $\tilde N$ in the $M_1$ topology of $\mathbb{D}([0,\infty))$, by using the CMT we just need to show that $h_2(N_n)=Z_n^{\leftarrow}$ for every $n\in\N$ and $h_2(\tilde N)=Z_H^{\leftarrow}$. We have, for $y\geq 0$ and $n\in\N$,
\[h_2 N_n(y)=\inf\{i/n:u_n^{-1}(X_i)<y\}=\inf\{i/n:u_n^{-1}(M_{i+1})<y\}=\inf\{t:Z_n(t)<y\}=Z_n^{\leftarrow}(y).\]

Now, let us see that the finite dimensional distributions of $h_2(\tilde N)$ are the same of $Z_H^{\leftarrow}$. For the unidimensional distribution, with $t,y\geq 0$ and $\bar H(\tau)=\e^{-\theta\tau}$, we have
\[\p(h_2\tilde N(y)\geq t)=\p(\tilde N([0,t)\times[0,y))=0)=\e^{-\theta t y}=\bar H^t(y).\]

For the bidimensional distribution, with $0\leq t_1<t_2$ and $y_1\geq y_2\geq 0$,
\[\p(h_2\tilde N(y_1)\geq t_1,h_2\tilde N(y_2)\geq t_2)=\p(\tilde N([0,t_1)\times[0,y_1))=0,\tilde N([t_1,t_2)\times[0,y_2))=0)\]
\[=\bar H^{t_1}(y_1)\bar H^{t_2-t_1}(y_2).\]
In case $0\leq y_1<y_2$,
\[\p(h_2\tilde N(y_1)\geq t_1,h_2\tilde N(y_2)\geq t_2)=\p(h_2\tilde N(y_2)\geq t_2)=\bar H^{t_2}(y_2),\]
so in general
\[\p(h_2\tilde N(y_1)\geq t_1,h_2\tilde N(y_2)\geq t_2)=\bar H^{t_1}(y_1\vee y_2)\bar H^{t_2-t_1}(y_2).\]

Following this pattern we get for the $k$-dimensional distribution
\begin{align*}
\p(h_2\tilde N(y_1)\geq t_1,\ldots,h_2\tilde N(y_k)\geq t_k)&=\bar H^{t_1}\left(\bigvee_{i=1}^k\{y_i\}\right)\bar H^{t_2-t_1}\left(\bigvee_{i=2}^k\{y_i\}\right)\cdots\bar H^{t_k-t_{k-1}}(y_k)\\
&=\p(Z_H(t_1)\geq y_1,\ldots,Z_H(t_k)\geq y_k)\\
&=\p(Z_H^{\leftarrow}(y_1)\geq t_1,\ldots,Z_H^{\leftarrow}(y_k)\geq t_k).
\end{align*}
Hence $h_2(\tilde N)=Z_H^{\leftarrow}$.

\end{proof}

\begin{remark}
We also have similar results for the multi-dimensional REPP process $N_n^*$. Let $b(r):=|B_r(\zeta)|$ and define the projections $h_1^*,h_2^*:M_p([0,\infty)\times T_{\zeta}(\X))\to\mathbb{D}([0,\infty))$ where, given the planar point process $m=\sum_{i=1}^{\infty}\delta_{(t_i,y_i)}$, $h_1^*m$ is the real valued function defined by
\[h_1^*m(t):=\begin{cases}
\inf\{b(\dist(y_i,\zeta)):t_i\leq t\} &\mbox{if } t>\underline{t}\\
\overline{y}^* &\mbox{if } t\leq\underline{t}
\end{cases},\]
with $\underline{t}=\inf\{(t_i)_{i=1}^{\infty}\}$ and $\overline{y}^*=\sup\{\inf\{b(\dist(y_i,\zeta)):t_i\leq t\}:t>\underline{t}\}$, and $h_2^*m$ is the real valued function defined by 
\[h_2^*m(y):=\begin{cases}
\inf\{t_i:b(\dist(y_i,\zeta))<y\} &\mbox{if } y>\underline{y}\\
\overline{t}^* &\mbox{if } y\leq\underline{y}
\end{cases},\]
with $\underline{y}=\inf\{(y_i)_{i=1}^{\infty}\}$ and $\overline{t}^*=\sup\{\inf\{t_i:b(\dist(y_i,\zeta))<y\}:y>\underline{y}\}$.

Hence, $h_1^*m$ and $h_2^*m$ are a.s.\ continuous in the $M_1$ topology of $\mathbb{D}([0,\infty))$ with respect to $N^*$ (note that $h_i^*=h_i\circ h, i=1,2$, where $h:M_p([0,\infty)\times T_{\zeta}(\X))\to M_p([0,\infty)^2)$ given by $h\left(\sum_{i=1}^{\infty}\delta_{(t_i,y_i)}\right)=\sum_{i=1}^{\infty}\delta_{(t_i,b(dist(y_i,\zeta)))}$ is a continuous function), so if $N_n^*\Rightarrow N^*$ we conclude by using the CMT that $Z_n^*\Rightarrow Z_H$ and $(Z_n^*)^{\leftarrow}\Rightarrow Z_H^{\leftarrow}$, where
\[Z_n^*(t):=h_1^*N_n^*(t)=\inf\left\{b\left(\dist\left(\frac{\Phi_\zeta^{-1}(T^j(x))}{b(g^{-1}(u_n(1)))^{1/d}},\zeta\right)\right):j\leq nt\right\}\]
\[(Z_n^*)^{\leftarrow}(y):=h_2^*N_n^*(y)=\inf\left\{j/n:b\left(\dist\left(\frac{\Phi_\zeta^{-1}(T^j(x))}{b(g^{-1}(u_n(1)))^{1/d}},\zeta\right)\right)<y\right\}=\inf\{t:Z_n^*(t)<y\}\]

We just need to see that $h_1^*(N^*)=Z_H$ and $h_2^*(N^*)=Z_H^{\leftarrow}$. For example, we have for the unidimensional distributions, with $t,y\geq 0$ and $\bar H(\tau)=\e^{-\theta\tau}$,
\begin{align*}
\p(h_1^*N^*(t)\geq y)&=\p(N^*([0,t]\times B_{b^{-1}(y)}(\zeta))=0)=\e^{-\theta t|B_{b^{-1}(y)}(\zeta)|}=\e^{-\theta t y}=\p(Z_H(t)\geq y),\\
\p(h_2^*N^*(y)\geq t)&=\p(N^*([0,t)\times B_{b^{-1}(y)}(\zeta))=0)=\p(Z_H(t)\geq y)=\p(Z_H^{\leftarrow}(y)\geq t).
\end{align*}
and the rest follows similarly the proofs of the previous propositions.
\end{remark}

\subsection{The record time and record value point processes}
Given the processes $Z_H(t)$ and $Z_H^{\leftarrow}(t)$, we next describe the distribution of their jump values, i.e. the locations of their discontinuities. This has natural application to the theory of record times and record values which we describe as follows. Consider the original sequences $(X_n)_{n\in\N_0}$, $(M_n)_{n\in\N}$, and let $t_1=0$. Define a strictly increasing sequence $(t_k)_{k\in\N}$ via:
\begin{equation}\label{eq.time-to-record}
t_k:=\inf\{j>t_{k-1}:X_j>M_j\}.
\end{equation}
Then this sequence $(t_k)_k$ forms the \emph{record times} associated to $M_n$, namely the times where $M_n$ jumps. The corresponding record values are given by the $X_{t_k}=M_{t_k+1}$. For the process $Z_n(t)=u_n^{-1}(M_{\lfloor nt\rfloor+1})$, we see that the jumps of $Z_n(t)$ occur precisely at the times $t_{k,n}:=t_k/n$ where $t_k$ is a record time. The jump values $Z_n(t_{k,n})$ are then the normalised record values $u_n^{-1}(X_{t_k})$. We consider the following two point processes defined on subsets of $\mathbb{R}$:
\begin{equation}
\mathcal{R}_n:=\sum_{j=1}^{\infty}\delta_{\frac{j}{n}}\cdot 1_{\{X_j>M_j\}},\quad
\mathcal{W}_n:=\sum_{k=1}^{\infty}\delta_{Z_n(t_{k,n})},
\end{equation}
the former is the \emph{record time process} and the latter is the \emph{record value process}. The weak convergence of these processes can be obtained from the weak convergence of the extremal processes, in the Skorokhod's $J_1$ topology, as in \cite{R87}, for example. 

To give an overview the construction of Skorokhod's $J_1$-topology consider first a metric inducing $J_1$ topology on the space $\mathbb{D}([a,b])$ given by:
\[d_{a,b}(\varphi_1,\varphi_2):=\inf_{h\in\Lambda}\left\{\|\varphi_1\circ h-\varphi_2\|\vee\|h-\mathrm{id}\|\right\},\]
where $\mathrm{id}$ is the identity mapping and the set $\Lambda$ is the collection of strictly increasing, continuous functions $h:[a,b]\to[a,b]$ such that $h(a)=a$ and $h(b)=b$. The construction carries over to $\mathbb{D}((0,\infty))$ by use of the following metric: let $r_{a,b}\varphi(x)$ denote the restriction of $\varphi(x)$ to the interval $[a,b]$, and define 
\[d_{0,\infty}(\varphi_1,\varphi_2):=\int_{0}^{1}\int_{1}^{\infty}e^{-t}(1\wedge d_{s,t}(r_{s,t}\varphi_1,r_{s,t}\varphi_2))\,dt\,ds.\]
Then convergence $\varphi_n\to\varphi$ in $\mathbb{D}((0,\infty))$ holds in the $J_1$ metric if $d_{0,\infty}(\varphi_n,\varphi)\to 0$ at each continuity point of $\varphi$.

We now state distributional results for the records' point processes using the weak convergence of extremal processes in the $J_1$ topology. We will apply this result to prove the convergence of record time and record value point processes in the particular case of clustering being created by a repelling periodic point as described in Section~\ref{subsubsec:2d-DS}. See Remark~\ref{rem:convergence-records-repelling}. 

\begin{theorem}
\label{thm.records}
Suppose that the processes $Z_n$ and $Z_n^{\leftarrow}$ converge weakly to $Z_H$ and $Z_H^{\leftarrow}$, respectively, in $\mathbb{D}((0,\infty))$ endowed with the Skorokhod's $J_1$ topology. Then, the point processes $\mathcal{R}_n$ and $\mathcal{W}_n$ converge weakly to a PRM with intensity $\gamma(t)=1/t$ on state space $(0,\infty)$, i.e. for any $0<a<b<\infty$,
\[\lim_{n\to\infty}\p(\mathcal{R}_n(a,b)=k)=\lim_{n\to\infty}\p(\mathcal{W}_n(a,b)=k)=\frac{a}{b}\cdot \frac{(\log(b/a))^k}{k!}.\]
\end{theorem}

We remark that:
\begin{itemize}
\item if the convergence of both $\mathcal{R}_n$ and $\mathcal{W}_n$ holds for a particular observable $\varphi:\X\to\mathbb{R}$, it also holds for any injective and monotone increasing transformation of $\varphi$;
\item the process $\mathcal{W}_n$ determines the jump times for the inverse process $Z_n^{\leftarrow}$.
\end{itemize}

\begin{proof}[Proof of Theorem \ref{thm.records}]
We first consider the process $\mathcal{R}_n$. Let $\widetilde{\mathbb{D}}((0,\infty))$ be the subset of $\mathbb{D}((0,\infty))$ consisting of functions which are constant between isolated jumps (i.e., the jumps do not accumulate anywhere in $(0,\infty)$). For an element $Y(t)\in\widetilde{\mathbb{D}}((0,\infty))$, let $h_3:\widetilde{\mathbb{D}}((0,\infty))\to M_p((0,\infty))$ be the counting function: $h_3(Y(t))=\sum_i\delta_{t_i}(0, t)$, where $t_i$ are jump times for $Y(t)$. As shown in \cite{R87} the function $h_3$ is a.s.\ continuous when restricted to functions on $\widetilde{\mathbb{D}}((0,\infty))$ in the $J_1$ topology. If $Y(t)=Y_G(t)$ is an extremal-G process, then $h_3(Y(t))$ is a PRM on $(0,\infty)$ with intensity $\gamma(t)=1/t$. This is also true for $Y(t)=Z_H(t)$ (note that, if $G=\bar H\circ f$ for some homeomorphism $f$, then $Y_G(t)$ has the same jump times of $Z_H(t)$). Hence to get the required convergence result we just need to apply the CMT.

Now consider the process $\mathcal{W}_n$ and the function $h_3$, but this time apply it to elements
of $Z_n^{\leftarrow}(t)\in\widetilde{\mathbb{D}}((0,\infty))$. We can see that $\mathcal{W}_n=h_3(Z_n^{\leftarrow})$ and that $Z_H$ has the same finite dimensional distributions of $S^{\leftarrow}(Y_{\Lambda})$, where $\Lambda(t)=\e^{-\e^{-t}}$, $S(x)=-\log(-\log(\bar H(x)))$ and $S^{\leftarrow}(t)=\inf\{x:S(x)<t\}$, so by the transformation theory for Poisson processes, the corresponding process $h_3(Z_H^{\leftarrow})$ is a PRM with mean measure $\lambda_W$ given by
\begin{align*}
\lambda_W([a,b])&:=|S(b)-S(a)|=\log(-\log(\bar H(b)))-\log(-\log(\bar H(a)))\\
&=\log(-\log(\e^{-\theta b}))-\log(-\log(\e^{-\theta a}))=\log(b/a),
\end{align*}
\ie a PRM with intensity $\gamma(t)=1/t$. Hence to get the required convergence we just need to apply the CMT.
\end{proof}

\begin{remark}
As before, let us assume the existence of normalising sequences $(a_n)_{n\in\N}\subset\R^+$ and $(b_n)_{n\in\N}\subset\R$ such that $n\p(a_n(X_0-b_n)>y)\to f(y)$ and $\p(a_n(M_n-b_n)\leq y)\to G(y)$, with $G=\bar H\circ f$ for some d.f. $G$ and $H$. If, instead of $\mathcal{W}_n$, we consider the process $\mathcal{V}_n:=\sum_k\delta_{Y_n(t_{k,n})}$, which determines the jump times for the inverse process $Y_n^{\leftarrow}(t):=\inf\{x:Y_n(x)>t\}$, then $\mathcal{V}_n$ converges weakly to the point process $\mathcal{V}$ on the domain of $G$, where $\mathcal{V}$ is a PRM with intensity measure $\lambda_V$ given by $\lambda_V([a,b]):=-\log(-\log G(b))+\log(-\log G(a))$, and in this case the limit process does depend on G, and hence on the form of the observable $\varphi$.
\end{remark}

\begin{remark}
\label{rem:convergence-records-repelling}
We consider the $J_1$ topology instead of the $M_1$ topology because, unlike what happened before with the functions $h_1$ and $h_2$, it is the suitable topology in which the new function $h_3$ is a.s.\ continuous. Note that in general convergence in $M_1$ does not imply convergence in $J_1$ and the convergence of the extremal processes that we obtained in Section~\ref{subsec:extremal-processes} was in the $M_1$ topology due to stacking on the vertical piles. However, in the setting considered in Section~\ref{subsubsec:2d-DS}, where $\zeta$ is a repelling periodic point, points of $N_n$ belonging to the same cluster can be disregarded except for the initial point (which is always the closest point to $\zeta$) without changing the projections $Z_n=h_1(N_n)$ and $Z_n^{\leftarrow}=h_2(N_n)$, and the same is true for the limiting process $\tilde N$ and its projections $Z_H=h_1(\tilde N)$ and $Z_H^{\leftarrow}=h_2(\tilde N)$ if we just consider the points for $\ell=0$. Hence, if we consider versions of these two-dimensional processes after deleting all but the first clustering mass point, in the case of $N_n$, and all the points in a vertical pile except the bottom one, in the case of $\tilde N$, we obtain point processes with the exact same projections by $h_1$ and $h_2$ of the original ones, but these new point processes (obtained after the deletions) allow to obtain the weak convergence of the extremal processes $Z_n$ and $Z_n^{\leftarrow}$ also valid in the $J_1$ topology, which in turn allows us to apply Theorem~\ref{thm.records} and obtain the convergence of the record point processes $\mathcal{R}_n$ and $\mathcal{W}_n$.
\end{remark}

We can also study the convergence of the record point processes by projecting directly from the space where the two-dimensional rare events $N_n$ are defined. This will allow us to see and understand from another perspective the convergence of the record point processes in the case of repelling periodic points considered in Section~\ref{subsubsec:2d-DS}, but also to give an example where there is no convergence.

We define 
\begin{align*}
  h:M_p([0,\infty)\times[0,\infty))&\longrightarrow M_p([0,\infty)) \\
  m=\sum_{i=1}^\infty \delta_{(t_i,z_i)}&\longmapsto h(m)=\sum_{i=1}^\infty \delta_{t_i}\I_{B_i}(z_i),
\end{align*}
where $B_i=\left[0,\inf\{z_j:\,t_j\leq t_i\}\right)$ if $t_i>\inf\{t_j:\,j\geq1\}$ and $B_i=\left[0,\inf\{z_j:\,t_j= \inf\{t_k:\,k\geq1\}\}\right]$, if $t_i=\inf\{t_j:\,j\geq1\}$.
Note that $h(N_n)=\mathcal{R}_n$.

For simple Radon point measures, $(m_n)_{n\in\N}$, vague convergence means that, as long as the limit process $m$ has no atoms on the boundary of a certain compact set, then the atoms of $m_n$ in that compact set converge to those of $m$ (see \cite[Proposition~3.13]{R87}).
Hence, it follows easily that $h$ is continuous at $m$ if $m$ satisfies the following properties: $m$ is simple, for all $\epsilon>0$ there exists $\gamma>0$ such that $m([0,\epsilon)\times[0,\gamma))>0$, $m([0,\infty)\times\{z\})\leq 1$ and $m(\{t\}\times[0,\infty))\leq 1$. 

Clearly, if we have no clustering, which means no piling on the vertical directions, then all point processes involved, including the limiting two-dimensional Poisson process $N$ live a.s. in a space of point measures $m$ with those properties and therefore, we can apply $h$ and CMT to obtain the convergence of record point process $\mathcal{R}_n$. In the presence of clustering, the limiting processes $\tilde N$, $N^\dag$, $\hat N$ or $\doublehat N$ do not satisfy the last property and indeed things can go wrong as the following example shows.

For each $n\in\N$, take $m_n=\delta_{(1-2/n,2)}+\delta_{(1,1)}$, which clearly converges in vague topology to $m=\delta_{(1,2)}+\delta_{(1,1)}$. But $h(m_n)=\delta_{1-2/n}+\delta_{1}$ does not converge vaguely to $h(m)=\delta_{1}$, since $h(m_n)([0,2))=2$ for all $n\in\N$ and $h(m)([0,2))=1$.

Observe that if the atoms of $m_n$ converged, in the first coordinate, from the right to $1$, say $m_n=\delta_{(1+2/n,2)}+\delta_{(1,1)}$, then we still would have that $h(m_n)$ converged to $h(m)$, in the vague topology. This is exactly what happens with $N_n$ in the cases described in Section~\ref{subsubsec:2d-DS}, Theorem~\ref{thm:convergence-hyperbolic} and Corollary~\ref{cor:two-dimensional}, where $\zeta$ is repelling periodic point, which means that a high observation is followed by increasingly smaller observations (which means increasingly higher respective frequencies $\tau$) within the same cluster, as well as in Examples~\ref{ex:2d} and \ref{ex:discontinuous}.

However, the same does not apply to Example~\ref{ex:exception}, where something similar to the non-continuity example occurs. Namely, if a record appears for $N_n$ that corresponds to a very high observation caused by the entrance of the orbit very close to $\pi/16$, then a new record observation occurs in the next iterate (corresponding to a $\tau$, which is approximately $3/10$ of the previous one). These two consecutive records are missed by the the projection of the limiting process, which only projects the bottom point of the two vertically aligned mass points. This shows that in this case, the convergence of the record point process cannot be established from the complete convergence of the two-dimensional process.

\bibliographystyle{abbrv}

\bibliography{Recordes}

\end{document}